\documentclass[final,3p,times]{elsarticle}

\usepackage[utf8]{inputenc}
\usepackage{amsmath}
\usepackage{amsfonts}
\usepackage{amssymb}
\usepackage{amsthm}
\usepackage{graphicx}
\usepackage{easybmat}
\usepackage{pifont} 
\usepackage{enumitem}
\usepackage{float}
\usepackage{subcaption}
\usepackage{todonotes}
\usepackage[final]{microtype}
\usepackage{mathtools}
\usepackage{url}
\mathtoolsset{showonlyrefs}

\biboptions{sort&compress}

\usepackage{color}
\def\Red#1{\textcolor{red}{#1}}

\usepackage{ragged2e}

\usepackage{natbib}

\newcommand{\sign}[1]{\mbox{sign}(#1)}

\newtheorem{theorem}{Theorem}
\newtheorem{corollary}{Corollary}
\newtheorem{definition}{Definition}
\newtheorem{example}{Example}
\newtheorem{lemma}{Lemma}

\newtheorem{remark}{Remark}

\newtheorem{assumption}[theorem]{Assumption}

\def\sign{\hskip2pt{\rm sign}\hskip0pt}

\journal{Journal of the Franklin Institute}
\begin{document}

\begin{frontmatter}



\title{On predefined-time consensus protocols for dynamic networks
\footnote{\Red{This is the author's version of the accepted manuscript: R. Aldana-López, D. Gómez-Gutiérrez, E. Jiménez-Rodríguez, J. D. Sánchez-Torres and A. G. Loukianov, “On predefined-time consensus for dynamic networks”, Journal of the Franklin Institute, 2019, ISSN: 0016-0032. DOI: 10.1016/j.jfranklin.2019.11.058.
Please cite the publisher's version. For the publisher's version and full citation details see:
\url{https://doi.org/10.1016/j.jfranklin.2019.11.058}
}}
}


\author[label0]{R.~Aldana-L\'opez}
\ead{rodrigo.aldana.lopez@gmail.com}

\author[label0,label1]{David~Gómez-Gutiérrez\corref{cor1}}
\ead{David.Gomez.G@ieee.org}
\cortext[cor1]{Corresponding Author.}

\author[label2,label3]{E.~Jim\'enez-Rodr\'iguez}
\ead{ejimenezr@gdl.cinvestav.mx}

\author[label3]{J.~D.~S\'anchez-Torres}
\ead{dsanchez@iteso.mx}

\author[label2]{A.~G.~Loukianov}
\ead{louk@gdl.cinvestav.mx}

\address[label0]{Multi-agent autonomous systems lab, Intel Labs, Intel Tecnología de M\'exico, Av. del Bosque 1001, Colonia El Bajío, Zapopan, 45019, Jalisco, M\'exico.}
\address[label1]{Tecnologico de Monterrey, Escuela de Ingenier\'ia y Ciencias, Av. General Ram\'on Corona 2514, Zapopan, 45201, Jalisco, M\'exico.}
\address[label2]{CINVESTAV, Unidad Guadalajara, Av. del Bosque 1145, colonia el Baj\'io, Zapopan , 45019, Jalisco, M\'exico.}
\address[label3]{Research Laboratory on Optimal Design, Devices and Advanced Materials -OPTIMA-, Department of Mathematics and Physics, ITESO, Perif\'erico Sur Manuel G\'omez Mor\'in 8585 C.P. 45604, Tlaquepaque, Jalisco, M\'exico.}

\begin{abstract}
This paper presents new classes of consensus protocols with fixed-time convergence, which enable the definition of an upper bound for consensus state as a parameter of the consensus protocol, ensuring its independence from the initial condition of the nodes. We demonstrate that our methodology subsumes current classes of fixed-time consensus protocols that are based on homogeneous in the bi-limit vector fields. Moreover, the proposed framework enables for the development of independent consensus protocols that are not needed to be homogeneous in the bi-limit. This proposal offers extra degrees of freedom to implement consensus algorithms with enhanced convergence features, such as reducing the gap between the real convergence moment and the upper bound chosen by the user. We present two classes of fixed-time consensus protocols for dynamic networks, consisting of nodes with first-order dynamics, and provide sufficient conditions to set the upper bound for convergence a priori as a consensus protocol parameter. The first protocol converges to the average value of the initial condition of the nodes, even when switching among dynamic networks. Unlike the first protocol, which requires, at each instant, an evaluation of the non-linear predefined time-consensus function, hereinafter introduced, per neighbor, the second protocol requires only a single evaluation and ensures a predefined time-consensus for static topologies and fixed-time convergence for dynamic networks. Predefined-time convergence is proved using Lyapunov analysis, and simulations are carried out to illustrate the performance of the suggested techniques. The exposed results have been applied to the design of predefined time-convergence formation control protocols to exemplify their main features. 
\end{abstract}


\end{frontmatter}

\section{Introduction}

Consensus algorithms allow a network of agents to agree on a value for its internal state in a distributed fashion by using only communication among neighbors~\cite{Olfati-Saber2007}. For this reason, they have attracted a great deal of attention in the field of automatic control, self-organizing systems and sensor networks~\cite{liu2017reliable}, with applications, for instance, to flocking~\cite{Olfati2006}, formation control~\cite{Oh2015,Ren2007,Li2013}, distributed resource allocation~\cite{Xu2017,Xu2017b}, distributed map formation~\cite{Aragues2012} and reliable filter design for sensor networks with random failures~\cite{liu2017reliable}.

For agents with first-order dynamics, a consensus protocol with asymptotic convergence to the average value of the initial conditions of the node has been proposed in~\cite{Olfati-Saber2007}. Using the stability results of the switching systems~\cite{Liberzon2003} it can be shown that such protocols reach a consensus even on dynamic networks by arbitrarily switching between highly connected graphs~\cite{Olfati-Saber2007,Cai2014}. Consensus protocols with enhanced convergence properties have been suggested based on finite-time~\cite{Bhat2005,Utkin1999}, and fixed-time~\cite{Polyakov2012,Parsegov2012} stability theory.

In~\cite{Sayyaadi2011,Shang2012,Wang2010,zhu2013,Franceschelli2013,Gomez-Gutierrez2018} continuous and discontinuous protocols with finite-time convergence were proposed. However, the convergence-time is an unbounded function of the initial conditions. A remarkable extension of the previous methods is the fixed-time convergent consensus. In this case, there exists a bound for the convergence time that is independent of the initial conditions \cite{Andrieu2008,Polyakov2012}. Therefore, for the design of high-performance consensus protocols, the fixed-time convergence is a desirable property. Several consensus protocol have been proposed based on the fixed-time stability results from~\cite{Polyakov2012,Parsegov2012}, see e.g.,~\cite{Zuo2014,Defoort2015,Parsegov2013,Sharghi2016,Hong2017,Ning2018,Wang2017b}. However, these consensus protocols have been justified only for static networks. Another fixed-time consensus algorithm was proposed by~\cite{Zuo2014a}, which is a consensus protocol for dynamic networks. However, similar to~\cite{Sharghi2016,Hong2017}, the convergence analysis is based on the upper estimate of the convergence-time given in~\cite{Polyakov2012}, which is known to be a too conservative upper bound~\cite{Aldana-Lopez2019}. 

Recently, to enable the application of fixed-time consensus algorithms in scenarios with time constraints, there has been an effort in finding the least upper bound of the settling time function of the class of fixed-time stable systems given in~\cite[Lemma~1]{Polyakov2012}. First, in~\cite{Parsegov2012} the least upper bound was found for a subclass of systems, which has lead to consensus protocols for dynamic networks such as~\cite{Parsegov2013,Zuo2014,Ni2017,Ning2017b,Wang2017b}. Recently, in~\cite{Aldana-Lopez2019} the least upper bound for the settling time was found for the general class of fixed-time stable systems given in~\cite[Lemma~1]{Polyakov2012}, based on this result in~\cite{AldanaConsensus2019} consensus protocols for dynamic networks, subsuming those given in~\cite{Ni2017,Ning2017b}, were presented, where an upper bound of the convergence time can be set a priori as a parameter of the protocol, in view of this feature, this class of consensus algorithms is known as predefined-time consensus.

Another approach to derive predefined-time consensus algorithms has been addressed via a linear function of the sum of the errors between neighboring nodes together with a time-varying gain, for instance, using time base generators~\cite{Morasso1997}, see e.g.,~\cite{Yong2012,Liu2018,Wang2017,Wang2018,Colunga2018b,Zhao2018,Ning2019}. However, these methods require that all nodes have a common time-reference because the same value of the time-varying gain should be applied to all nodes. Thus, this approach is not suitable in GPS-denied environments or in scenarios where having a common time reference is a strong assumption. Moreover, such time-varying gain becomes singular at the pre-set time, either because the gain goes to infinite as the time tends to the pre-set time~\cite{Yong2012,Zhao2018} or because it produces Zeno behavior (infinite number of switching in a finite-time interval) as the time tends to the pre-set time~\cite{Liu2018}.

\section{Contribution}
This paper aims to provide new classes of consensus protocols, to obtain fixed-time convergence in dynamic networks arbitrarily switching among connected topologies. Sufficient conditions are derived, such that the upper bound for the convergence time is selected a priori as a parameter of the protocol. Such protocols are referred as predefined-time consensus protocols.

Two classes of consensus protocols for networks composed of nodes with first-order dynamics are proposed. The first one is presented to solve the average-consensus problem with predefined convergence under dynamic networks. The second protocol is shown to have predefined-time convergence on static networks and fixed-time convergence under dynamic networks. Unlike the first protocol that requires, at each time instant, one evaluation of the nonlinear predefined-time consensus function (hereinafter introduced) per neighbors, the second protocol only requires a single evaluation, with the trade-off of not ensuring the convergence to the average value of the nodes' initial conditions.

Contrary to consensus protocols with predefined convergence based on time-varying gains as in~\cite{Yong2012,Liu2018,Wang2017,Wang2018,Colunga2018b}, the proposed classes of protocols does not require the strong assumption of a common time-reference for all nodes. Moreover, unlike autonomous fixed-time consensus protocols~\cite{Ning2017b,Wang2017b}, which are based on a subclass of the fixed-time stable systems given in~\cite[Lemma~1]{Polyakov2012}, which uses homogeneous in the bi-limit~\cite{Andrieu2008} vector fields, in this paper a methodology for the design of new consensus protocols is presented, showing that predefined-time consensus can be achieved with a broader class of consensus protocols, that are not required to be homogeneous in the bi-limit. This result provides extra degrees of freedom to select a protocol, for instance, to reduce the slack between the predefined upper bound for the convergence and the exact convergence time. This methodology generalizes the recent results~\cite{Ning2017b,AldanaConsensus2019} on fixed-time consensus for dynamic networks formed by agents with first-order dynamics.

The rest of the paper is organized as follows. Section~\ref{Sec.Preliminaries} introduces the preliminaries on graph theory and predefined-time stability. Section~\ref{Sec.MainResult} presents two new classes of consensus protocols with predefined-time convergence together with illustrative examples showing the performance of the proposed approach.  In Section~\ref{Sec:Formation} these results are applied to the design of formation control protocols with predefined-time convergence.
Finally, Section~\ref{Sec.Conclu} provides the concluding remarks and discusses future work.

\section{Preliminaries}
\label{Sec.Preliminaries}

\subsection{Graph Theory}
\label{SubSec.GraphTheory}
The following notation and preliminaries on graph theory are taken mainly from~\cite{godsil2001}.

An undirected graph $\mathcal{X}$ consists of a vertex set $\mathcal{V}(\mathcal{X})$ and an edge set $\mathcal{E}(\mathcal{X})$ where an edge is an unordered pair of distinct vertices of $\mathcal{X}$. Writing $ij$ denotes an edge, and $j\sim i$ denotes that the vertex $i$ and vertex $j$ are adjacent or neighbors, i.e., there exists an edge $ij$. The set of neighbors vertex of $i$ in the graph $\mathcal{X}$ is expressed by $\mathcal{N}_i(\mathcal{X})=\{j:ji\in \mathcal{E}(\mathcal{X})\}$. A path from $i$ to $j$ in a graph is a sequence of distinct vertices starting with $i$ and ending with $j$ such that consecutive vertices are adjacent. If there is a path between any two vertices of the graph $\mathcal{X}$ then $\mathcal{X}$ is said to be connected. Otherwise, it is said to be disconnected.  

A weighted graph is a graph together with a weight function $\mathcal{W}:\mathcal{E}(\mathcal{X})\to \mathbb{R}_{+}$.
If $\mathcal{X}$ is a weighted graph such that $ij\in\mathcal{E}(\mathcal{X})$ has weight $a_{ij}$ and  $n=|\mathcal{V}(\mathcal{X})|$. Then the incidence matrix $D(\mathcal{X})$ is a $\vert \mathcal{V}(\mathcal{X})\vert \times \vert \mathcal{E}(\mathcal{X})\vert$ matrix, such that if $ij\in \mathcal{E}(\mathcal{X})$ is an edge with weight $a_{ij}$ then the column of $D$ corresponding to the edge $ij$ has only two nonzero elements: the $i-$th element is equal to $\sqrt{a_{ij}}$ and the $j-$th element is equal to $-\sqrt{a_{ij}}$. Clearly, the incidence matrix $D(\mathcal{X})$, satisfies $\mathbf{1}^TD(\mathcal{X})=0$. The Laplacian of $\mathcal{X}$ is denoted by $\mathcal{Q}(\mathcal{X})$ (or simply $\mathcal{Q}$ when the graph is clear from the context) and is defined as $\mathcal{Q}(\mathcal{X})=D(\mathcal{X})D(\mathcal{X})^T$. The Laplacian matrix $\mathcal{Q}(\mathcal{X})$ is a positive semidefinite and symmetric matrix. Thus, its eigenvalues are all real and non-negative.

When the graph $\mathcal{X}$ is clear from the context we omit $\mathcal{X}$ as an argument. For instance we write $Q$, $D$, etc to represent the Laplacian, the incidence matrix, etc.

\begin{lemma}~\cite{godsil2001}
\label{lemma:Lambda2}
Let $\mathcal{X}$ be a connected graph and $\mathcal{Q}$ its Laplacian. The eigenvalue $\lambda_1(\mathcal{Q})=0$ has algebraic multiplicity one with eigenvector $\mathbf{1}=[1\ \cdots\ 1]^T$. The smallest nonzero eigenvalue of $\mathcal{Q}$, denoted by $\lambda_2(\mathcal{Q})$ satisfies that $\lambda_2(\mathcal{Q})=\underset{x\perp \mathbf{1},x\neq 0}{\min}\dfrac{x^T \mathcal{Q}x}{x^Tx}$. 
\end{lemma}

It follows from Lemma~\ref{lemma:Lambda2} that for every $x\bot\mathbf{1}$, $x^T\mathcal{Q}x\geq \lambda_2(\mathcal{Q}) \Vert x \Vert_2^2>0$. $\lambda_2(\mathcal{Q}(\mathcal{X}))$ is known as the algebraic connectivity of the graph $\mathcal{X}$. 

\begin{definition}
A switched dynamic network $\mathcal{X}_{\sigma(t)}$ is described by the ordered pair $\mathcal{X}_{\sigma(t)}=\langle\mathcal{F},\sigma\rangle$ where $\mathcal{F}=\{\mathcal{X}_1,\ldots,\mathcal{X}_m\}$ is a collection of graphs having the same vertex set and $\sigma:[t_0,\infty)\rightarrow \{1,\ldots m\}$ is a switching signal determining the topology of the dynamic network at each instant of time. 
\end{definition}

In this paper, we assume that $\sigma(t)$ is generated exogenously and that there is a minimum dwell time between consecutive switchings in such a way that Zeno behavior in network's dynamic is excluded, i.e., there is a finite number of switchings in any finite interval. Notice that, no maximum dwell time is set, thus the system may remain under the same topology during its evolution.

\subsection{Fixed-time stability with predefined upper bound for the settling time}
The preliminaries on predefined-time stability are taken from~\cite{aldana2019design}.

Consider the system
\begin{equation}\label{eq:sys}
    \dot{x}=-\frac{1}{T_c}f(x), \ \forall t\geq t_0, \ f(0)=0,\ \ x(t_0)=x_0,
\end{equation}
where $x\in\mathbb{R}^n$ is the state of the system, $T_c>0$ is a parameter and $f:\mathbb{R}^n\to\mathbb{R}^n$ is nonlinear, continuous on $x$ everywhere except, perhaps, at the origin. 

We assume that $f(\cdot)$
is such that the origin of~\eqref{eq:sys} is asymptotically stable and, except at the origin, \eqref{eq:sys} has the properties of existence and uniqueness of solutions in forward-time on the interval $[t_0,+\infty)$.
The solution of \eqref{eq:sys} with initial condition $x_0$ is denoted by $x(t;x_0)$.

\begin{definition}\cite{Polyakov2014}(Settling-time function)
The \textit{settling-time function} of system~\eqref{eq:sys} is defined as
$T(x_0,t_0)=\inf\{\xi\geq t_0: \lim_{t\to\xi}x(t;x_0)=0\}-t_0$.
\end{definition}

\begin{definition}\cite{Polyakov2014} \label{def:fixed}(Fixed-time stability) 
System \eqref{eq:sys} is said to be \textit{fixed-time stable} if it is asymptotically stable~\cite{Khalil2002} and the settling-time function $T(x_0,t_0)$ is bounded on  $\mathbb{R}^n\times\mathbb{R}_+$, i.e. there exists $T_{\text{max}}\in\mathbb{R}_+\setminus\{0\}$ such that $T(x_0,t_0)\leq T_{\text{max}}$ if $t_0\in\mathbb{R}_+$ and $x_0\in\mathbb{R}^n$. Thus, $T_{\text{max}}$ is an Upper Bound of the Settling Time (\textit{UBST}) of $x(t;x_0)$.
\end{definition}

\begin{assumption}
\label{Assump:AsympSys}
Let $\Psi(z)=\Phi(|z|)^{-1}\sign(z)$, with $z\in\mathbb{R}$, where  $\Phi:\mathbb{R}_+\to\Bar{\mathbb{R}}_+\setminus\{0\}$ is a function satisfying $\Phi(0)=+\infty$, $\forall z\in\mathbb{R}_+\setminus\{0\}$, $\Phi(z)<+\infty$ and
\begin{equation}
\label{Eq:Finite_Improper}
     \int_0^{+\infty} \Phi(z)dz = 1.
\end{equation}

\end{assumption}

\begin{lemma}\label{Lemma:TimeScale} 
Let $\Psi(z)$ and be a function satisfying Assumption~\ref{Assump:AsympSys}, then, the system
\begin{equation}
\label{Eq:TSFunc}
    \dot{x}=-\frac{1}{T_c}\Psi(x), \ \ x(t_0)=x_0,
\end{equation}
is asymptotically stable and the least \textit{UBST} function $T(x_0)$ is given by
\begin{equation}
\label{Eq:Time_integral1}    
\sup_{x_0 \in \mathbb{R}^n} T(x_0)=T_c.
\end{equation}
\end{lemma}

\begin{theorem} (Lyapunov characterization for fixed-time stability with predefined \textit{UBST})
\label{thm:weak_pt} 
If there exists a continuous positive definite radially unbounded function $V:\mathbb{R}^n\to\mathbb{R}$, such that its time-derivative along the trajectories of~\eqref{eq:sys} satisfies
\begin{equation}\label{eq:dV_weak}
\dot{V}(x)\leq-\frac{1}{T_c}\Psi(V(x)),  \ \  x\in\mathbb{R}^n\setminus\{0\},
\end{equation}  
where $\Psi(z)$ satisfies Assumption~\ref{Assump:AsympSys}, then, system~\eqref{eq:sys} is fixed-time stable with $T_c$ as the predefined \textit{UBST}. 
\end{theorem}

\section{Main Result}
\label{Sec.MainResult}

It is assumed a multi-agent system composed of $n$ agents, whcih are able to communicate with its neighbors according to a communication topology given by the switching dynamic networks $\mathcal{X}_{\sigma(t)}$.
The $i-$th agent dynamics is given by
$$
\dot{x}_i=u_i
$$
where $u_i$ is called the consensus protocol.  The aim of the paper is to introduce new classes of consensus protocols for dynamic networks as well as to provide the conditions under which, using only information from the neighbors $\mathcal{N}_i(\mathcal{X}_{\sigma(t)})$, the convergence is guaranteed in a predefined-time. 

\begin{definition}
\label{Def:ConsensusFunction}
Let $\Omega:\mathbb{R}\to\mathbb{R}$ be a monotonically increasing function 
satisfying 
\begin{itemize}
    \item there exist a function $\hat{\Omega}:\mathbb{R}_+\to\mathbb{R}_+$, a non-increasing function $\beta:\mathbb{N}\to\mathbb{R}_+$ and $d\geq 1$, such that for all $x=(x_1,\dots,x_n)^T$, $x_i\in\mathbb{R}_+$, the inequality
    \begin{equation}
    \hat{\Omega}\left(\beta(n) \|x\|_2\right)\leq\beta(n)^d\sum_{i=1}^n\Omega(x_i)
    \label{convex_degree}
\end{equation}
    holds, where $\|x\|_2 =  \left(\sum_{i=1}^n|x_i|^2\right)^{1/2}$.
    
    \item $\Psi(z) = z^{-1}\hat{\Omega}(|z|)$ satisfies Assumption~\ref{Assump:AsympSys}.

\end{itemize}
then, $\Omega(\cdot)$ is called a predefined-time consensus function.
\end{definition}

\begin{lemma}
\label{lemma:convex_subhomo}
Let $\Omega:\mathbb{R}\to\mathbb{R}$ be a monotonically increasing function, then if it satisfies either
\begin{itemize}
    \item \textit{i}) $\Omega(x+y)\leq\Omega(x)+\Omega(y)$, i.e. $\Omega(z)$  sub-additive, and $\Omega(\mu x)\leq\mu^d\Omega(x)$ for $\mu\in[0,1]$ and $d\geq 1$, i.e. $\Omega(z)$  sub-homogeneous of degree $d$.
    \item \textit{ii}) $\Omega(z)$ convex.
\end{itemize}
Then, $\Omega(z)$  complies with \eqref{convex_degree} for  $\beta(n)=n^{-1}$, $\hat{\Omega}(z)=\Omega(z)$ and $d$ the degree of sub-homogeneity for \textit{i}) and $d=1$ for \textit{ii}).
\end{lemma}
\begin{proof}
Lemma \ref{Lemma:Hardy} leads to $\Omega(n^{-1}\|x\|_2)\leq\Omega(n^{-1}\|x\|_1) = \Omega\left(n^{-1}\sum_{i=1}^{n}x_i\right)$. Moreover, for \textit{i}) $\Omega\left(n^{-1}\sum_{i=1}^{n}x_i\right)\leq \sum_{i=1}^n\Omega(n^{-1}x_i)\leq n^{-d}\sum_{i=1}^n\Omega(x_i)$. For \textit{ii}), $\Omega\left(n^{-1}\sum_{i=1}^{n}x_i\right)\leq n^{-1}\sum_{i=1}^n\Omega(x_i)$ due to Jensen's inequality of convex functions \cite[Formula 5]{jensen1906}. Hence, $\Omega(z) = \hat{\Omega}(z)$ complies with \eqref{convex_degree} for either \textit{i}) or \textit{ii}). 
\end{proof}

\begin{lemma}
\label{Lemma:ExamplesFunc}
The following functions are predefined-time consensus functions, satisfying~\eqref{convex_degree} with $d=1$, $\beta(n)=\frac{1
}{n}$ and $\hat{\Omega}(z)=\Omega(z)$:
\begin{itemize}
    \item {\it i)} $\Omega(z) = \frac{1}{p}\exp(z^p)z^{2-p}$ for $0<p\leq 1$
    \item {\it ii)} $\Omega(z) = \frac{\pi}{2}(\exp(2z) - 1)^{1/2}z$
    \item {\it iii)} $\Omega(z) = \gamma z(a z^p+b z^q)^k$ where $a,b,p,q,k>0$ satisfy $kp<1$ and $kq>1$, and \[\gamma=\frac{\Gamma \left(m_p\right) \Gamma \left(m_q\right)}{a^{k}\Gamma (k) (q-p)}\left(\frac{a}{b}\right)^{m_p},\] with $m_p=\frac{1-kp}{q-p}$, $m_q=\frac{kq-1}{q-p}$ and $\Gamma(\cdot)$ is the Gamma function defined as $\Gamma(z)=\int_0^{+\infty} e^{-t}t^{z-1}dt$~\cite[Chapter~1]{Bateman1955}.
\end{itemize}
\end{lemma}

\begin{proof}
For item {\it i)}, note that $\frac{d^2}{dz^2}\Omega(z) = \exp(z^p)z^{-p}(p^2z^{2p}+(3p-p^2)z^p+2-(3p-p^2))$. Moreover, note that $0<(3p-p^2)\leq 2$. Henceforth, $\frac{d^2}{dz^2}\Omega(z)\geq0$ and therefore $\Omega(z)$ is convex. For item {\it ii)} note that $\frac{d^2}{dz^2}\Omega(z) = (2/\pi)\exp(2z)(\exp(2z)-1)^{-2/3}(\exp(2z)z-2z+2\exp(2z)-2)$. Recall that $\exp(z)\geq 1+z$ for $z\geq 0$. Hence, $\exp(2z)z-2z+2\exp(2z)-2\geq z + z^2\geq 0$ and therefore $\Omega(z)$ is convex. For {\it iii)} it was proved in \cite{AldanaConsensus2019} that $\Omega(z)$ is convex. Therefore, by Lemma \ref{lemma:convex_subhomo}, $\Omega(z)$ and satisfy \eqref{convex_degree} $d=1$ and $\hat{\Omega}(z)=\Omega(z)$ for items \textit{i)-iii)}.

Moreover, $\Psi(z) = p\exp(-z^p)z^{1-p}$, $\Psi(z) = \frac{2}{\pi}(\exp(z)-1)^{1/2}$ and $\Psi(z) = \gamma^{-1}(a x^p + bx^q)^{-k}$ satisfy the conditions of Assumption~\ref{Assump:AsympSys}, as shown in~\cite{aldana2019design}. 
Therefore, the functions in items {\it i)}--{\it iii)} are predefined-time consensus functions.
\end{proof}

\begin{remark}
\label{Remark:ReduceSlack}
In the following, we derive the condition for fixed- and predefined-time consensus under dynamic networks, with protocols that extend those presented in the literature, for instance, ~\cite{Ning2017b,Wang2017b}. In the interest of providing a general result we may obtain $\beta(\cdot)$ and $\hat{\Omega}(\cdot)$, resulting in satisfying~\eqref{convex_degree} in a conservative manner. However, in some scenarios, $\beta(\cdot)$ can be obtained such that~\eqref{convex_degree} is less conservative, resulting in protocols where the slack between the true convergence and the predefined one is reduced. The following lemma illustrates this case for $k=1$ in the predefined-time consensus function given in Lemma~\ref{Lemma:ExamplesFunc} item \textit{iii)}. 
\end{remark}

\begin{lemma}
\label{Lemma:RhoLarge}
The function $\Omega(z) = \gamma \left( a z^{p+1}+bn^{\frac{q-1}{2}} z^{q+1}\right)$ where $a,b,p,q>0$ satisfy $p<1$ and $q>1$, and \[\gamma=\frac{\Gamma \left(m_p\right) \Gamma \left(m_q\right)}{a (q-p)}\left(\frac{a}{b}\right)^{m_p},\] with $m_p=\frac{1-p}{q-p}$, $m_q=\frac{q-1}{q-p}$ and $\Gamma(\cdot)$ the Gamma function, is a predefined-time consensus function, satisfying~\eqref{convex_degree} with $\beta(n)=1$ and $\hat{\Omega}(z)=\gamma(a z^{p+1} + b z^{q+1})$.
\end{lemma}
\begin{proof}
Let $x_i\in\mathbb{R}_+$ and note that from Lemma~\ref{Lemma:Hardy} it follows that $\sum_{i=1}^n x_i^{1+p} = \sum_{i=1}^n(x_i^2)^\frac{1+p}{2}=\|(x_1^2,\dots,x_n^2)^T\|_{(1+p)/2}^{(1+p)/2}\geq\|(x_1^2,\dots,x_n^2)^T\|_{1}^{(1+p)/2}\geq\left(\sum_{i=1}^nx_i^2\right)^{\frac{1+p}{2}}=\|x\|_2^{1+p}$. Similarly, $\sum_{i=1}^n x_i^{1+q} = \sum_{i=1}^n(x_i^2)^\frac{1+q<}{2}=\|(x_1^2,\dots,x_n^2)^T\|_{(1+q)/2}^{(1+q)/2}\geq n^{\frac{1-q}{2}}\|(x_1^2,\dots,x_n^2)^T\|_{1}^{(1+q)/2}\geq n^{\frac{1-q}{2}}\left(\sum_{i=1}^nx_i^2\right)^{\frac{1+q}{2}}=n^{\frac{1-q}{2}}\|x\|_2^{1+q}$. Hence, $\sum_{i=1}^n\Omega(x_i) = a\sum_{i=1}^nx_i^{1+p}+bn^\frac{q-1}{2}\sum_{i=1}^nx_i^{1+q}\geq a\|x\|^{1+p}_2 + b\|x\|_2^{1+q}=\hat{\Omega}(\|x\|_2)$. The proof that $\Psi(z)=z^{-1}\hat{\Omega}(|z|)$ satisfies Assumption~\ref{Assump:AsympSys} can be found in~\cite{aldana2019design}.
\end{proof}

\begin{remark}
\label{Remark:Symmetric}
In~\cite{Zuo2014,Ning2017b,Wang2017b}, fixed-time consensus protocols were proposed based on the function $\Omega(\cdot)$ given in Lemma~\ref{Lemma:RhoLarge}, but restricted to the case where $p=1-s$ and $q=1+s$ with $0<s<1$. Notice that, in Lemma~\ref{Lemma:RhoLarge} such restriction is removed. Moreover, in this paper we show that predefined-time consensus can be obtained with a larger class of functions, such as those given in Lemma~\ref{Lemma:ExamplesFunc}.
\end{remark}

Based on predefined-time consensus functions $\Omega(\cdot)$ the following classes of consensus protocols for dynamic networks are proposed:
\begin{enumerate}[label=(\roman*)]
    \item 
\begin{equation}
\label{Eq:ConsensusProtocolA} u_i=\kappa_i\sum_{j\in\mathcal{N}_i(\mathcal{X}_{\sigma(t)})}\sqrt{a_{ij}}e_{ij}^{-1}\Omega(|e_{ij}|), \ \ \ \ \ e_{ij}=\sqrt{a_{ij}}(x_j(t)-x_i(t)),
\end{equation}
\item
\begin{equation}
\label{Eq:ConsensusProtocolB}
u_i=\kappa_ie_i^{-1}\Omega(|e_i|), \ \ \ \ \ e_i=\sum_{j\in\mathcal{N}_i(\mathcal{X}_{\sigma(t)})}a_{ij}(x_j-x_i).
\end{equation}
\end{enumerate}

We show that, if parameters $\kappa_i$ satisfies $\kappa_i>0$ then consensus is achieved with fixed-time convergence. Moreover, we derive the condition on $\kappa_i$ under which predefined-time convergence is obtained.

\subsection{Predefined-time average consensus for dynamic networks}

In this subsection we focus on the analysis of the class of consensus protocol~\eqref{Eq:ConsensusProtocolA}, we will derive the condition under which consensus on the average of the initial values of the agents is achieved.
Notice that, the dynamics of the network under these protocols can be written as
\begin{equation}
\label{ConsensusDynamicA}
\dot{x}=-D(\mathcal{X}_{\sigma(t)})\mathcal{F}(D(\mathcal{X}_{\sigma(t)})^Tx),
\end{equation}
where, for $z=[z_1 \ \cdots \ z_n]^T\in\mathbb{R}^n$, the function $\mathcal{F}:\mathbb{R}^n\rightarrow \mathbb{R}^n$ is defined as
\begin{equation}
\label{Eq:Fp}    
\mathcal{F}(z)=
\begin{bmatrix}
\kappa_1z_1^{-1}\Omega(|z_1|) \\ 
\vdots \\
\kappa_nz_n^{-1}\Omega(|z_n|)
\end{bmatrix}
\end{equation}

To prove that~\eqref{Eq:ConsensusProtocolA} is a predefined-time average consensus algorithm, the following result will be used.

\begin{lemma}
\label{DeltaOrth}
Let the disagreement variable $\delta$ be such that $x = \alpha\mathbf{1} + \delta$, where $\alpha=\frac{1}{n}\mathbf{1}^Tx(t_0)$ is the average value of the nodes' initial condition. Then, if the graph is connected, under the consensus protocol~\eqref{Eq:ConsensusProtocolA}, $\delta^T \mathbf{1}=0$.
\end{lemma}
\begin{proof}
Let $s_x=\mathbf{1}^Tx$ be the sum of the nodes' values. Recall that  $\mathbf{1}^TD(\mathcal{X}_{\sigma(t)})=0$, then $\dot{s}_x=-\mathbf{1}^T D(\mathcal{X}_{\sigma(t)})F_p(D(\mathcal{X}_{\sigma(t)})^Tx)=0$. Thus, $s_x$ is constant during the evolution of the system, i.e. $\forall t\geq 0$, $s_x(t)=\mathbf{1}^Tx(t_0)=s_x(t_0)$. Therefore, 
$$
\mathbf{1}^T\delta=\mathbf{1}^Tx-\alpha n=s_x(t)-s_x(t_0)=0, \ \ \text{ for all } \ t\geq t_0.
$$
\end{proof}

\begin{theorem} (Predefined-time average consensus for fixed and dynamic networks)
\label{Th:ConsensusA}
Let $\mathcal{X}_{\sigma(t)}=\langle\mathcal{F},\sigma\rangle$ be a switched dynamic network formed by strongly connected graphs, and let $\Omega(\cdot)$ be a predefined-time consensus function with associated $d$, $\beta(\cdot)$  and $\hat{\Omega}(\cdot)$ such that~\eqref{convex_degree} holds. Then, if $\kappa_i>0$, $i=1,\ldots,n$, then~\eqref{Eq:ConsensusProtocolA} is a consensus protocol with fixed-time convergence. Moreover, if $\kappa_i \geq \dfrac{\beta(\underline{m})^{d}}{\lambda\beta(\overline{m})^{2}T_c}$, $i=1,\ldots,n$, where 
\begin{equation}
\label{Eq:LambdaW}
\lambda = \min_{\mathcal{X}_i\in\mathcal{F}} \lambda_2(\mathcal{Q}(\mathcal{X}_i)), \ \ \ \underline{m}=\min_{\mathcal{X}_i\in\mathcal{F}}|\mathcal{E}(\mathcal{X}_i)|
\text{ and }
\overline{m} = \max_{\mathcal{X}_i\in\mathcal{F}}|\mathcal{E}(\mathcal{X}_i)|.
\end{equation}
then, ~\eqref{Eq:ConsensusProtocolA} is an average consensus algorithm for dynamic networks with predefined convergence time bounded by $T_c$, i.e. all trajectories of~\eqref{ConsensusDynamicA} converge to the average of the initial conditions of the nodes in a time $T(x_0)\leq T_c$.
\end{theorem}

\begin{proof}
Let $\delta = [\delta_1, \dots , \delta_n]^T$ be the disagreement variable $x(t) = \alpha\mathbf{1} + \delta(t)$, where $\alpha=\frac{1}{n}\mathbf{1}^Tx_0$, which by Lemma~\ref{DeltaOrth} satisfies $\mathbf{1}^T\delta = 0$. Note that $\dot{x} = \dot{\delta} = -D(\mathcal{X}_l)\mathcal{F}(D(\mathcal{X}_l)^T\delta)$. Consider the Lyapunov function candidate
\begin{equation}
V(x) = \sqrt{\lambda}\beta(\overline{m}) \|\delta\|.
\label{Eq:LyapunovProt1_0}
\end{equation}
which is radially unbounded and satisfies $V(x)=0$ if and only if $\delta=0$.

To show that consensus is achieved on dynamic networks under arbitrary switchings, we will prove that~\eqref{Eq:LyapunovProt1_0} is a common Lyapunov function for each subsystem of the switched nonlinear system~\eqref{ConsensusDynamicB}~\cite[Theorem 2.1]{Liberzon2003}. To this aim, assume that $\sigma(t)=l$ for $t\in[0,\infty)$.
Then, it follows that
\begin{equation*}
\dot{V}(x) = \frac{\sqrt{\lambda}\beta(\overline{m})}{\|\delta\|}\delta^T\dot{\delta}  = -\lambda\beta(\overline{m})^2V^{-1}\delta^TD(\mathcal{X}_l)\mathcal{F}(D(\mathcal{X}_l)^T\delta),
\end{equation*}
Let $v = D(\mathcal{X}_l)^T\delta = [v_1,\dots,v_m]^T$, therefore:
\begin{align}
\dot{V}(x) = -\lambda\beta(\overline{m})^2V^{-1} v^T\mathcal{F}(v) 
  = -\lambda\beta(\overline{m})^2V^{-1}\sum_{i=1}^m\kappa_i\Omega\left(|v_i|\right).\label{prot1_firsteq}
\end{align}
Using the fact that $\Omega(\cdot)$ is a predefined time consensus function, the right hand side of \eqref{prot1_firsteq} can be rewritten as:
\begin{align*}
\lambda\beta(\overline{m})^{2}\beta(m)^{-d}V^{-1}\sum_{i=1}^m\beta(m)^ d\kappa_i\Omega\left(|v_i|\right) 
\geq &
\kappa\lambda\beta(\overline{m})^{2}\beta(m)^{-d}V^{-1}\sum_{i=1}^m\beta(m)^d\Omega\left(|v_i|\right) \\
\geq &\kappa\lambda\beta(\overline{m})^{2}\beta(m)^{-d}V^{-1}\hat{\Omega}\left(\beta(m)\|v\|_2\right),
\end{align*}
where $\kappa = \min\{\kappa_1,\dots,\kappa_n\}$.
Moreover, it follows from Lemma~\ref{Lemma:Hardy} and Lemma \ref{lemma:Lambda2} that
\begin{equation*}
 \|v\|_2 = \sqrt{v^Tv} = \sqrt{\delta^T \mathcal{Q}(\mathcal{X}_l)\delta} \geq \sqrt{\lambda}\|\delta\|=\beta(\overline{m})^{-1}V.
\end{equation*}
Therefore:
\begin{align}
\label{prot1_res1}
V^{-1}\hat{\Omega}\left(\beta(m)\|v\|_2\right) \geq
V^{-1}\hat{\Omega}\left(\frac{\beta(m)}{\beta(\overline{m})}V\right)\geq 
V^{-1}\hat{\Omega}\left(V\right)  = \Psi(V)
\end{align}
Moreover, the following inequality is obtained from \eqref{prot1_firsteq}, and \eqref{prot1_res1}:
\begin{equation}
\dot{V}(x) \leq -\kappa\lambda\beta(\overline{m})^{2}\beta(m)^{-d}\Psi(V) 
\label{Eq:Predef1}
\end{equation}
Then, according to Theorem~\ref{thm:weak_pt}, the disagreement variable $\delta$ converges to zero in a fixed-time upper bounded by $T_c$, and therefore protocol~\eqref{Eq:ConsensusProtocolA}   guarantees that the consensus is achieved in a fixed-time upper bounded by
$$
\sup_{x_0 \in \mathbb{R}^n} T(x_0)\leq\frac{\beta(m)^{d}}{\kappa\lambda\beta(\overline{m})^{2}}
$$
Therefore, if 
\begin{equation}
\label{Eq:KappaProtA}    
\kappa_i\geq\kappa=\dfrac{\beta(\underline{m})^{d}}{\lambda\beta(\overline{m})^{2}T_c},
\end{equation}
$i=1,\ldots,n$, then
\begin{equation}
\dot{V}(x) \leq -\dfrac{\beta(\underline{m})^{d}}{\beta(m)^{d}T_c}\Psi(V) \leq -\frac{1}{T_c}\Psi(V)
\end{equation}
and, since $\beta(\cdot)$ is a non-increasing function, then $V(x)$ converges to zero in a predefined-time upper bounded by $T_c$.

Since the above argument holds for any connected $\mathcal{X}_l\in\mathcal{F}$, then protocol~\eqref{Eq:ConsensusProtocolA} guarantees that the consensus is achieved, before a predefined-time $T_c$, on switching dynamic networks under arbitrary switching. Furthermore, it follows from Lemma~\ref{DeltaOrth} that the consensus state is the average of the initial values of the agents. 
\end{proof}

\begin{example}
\label{Example1}
Consider a network composed of 10 agent and four different communication topologies, $\mathcal{F}=\{\mathcal{X}_1,\ldots,\mathcal{X}_4\}$ as shown in Figure~\ref{fig:DFD_net1}-\ref{fig:DFD_net4} with algebraic connectivity $\lambda_2(\mathcal{Q}(\mathcal{X}_1))=0.2279$,    $\lambda_2(\mathcal{Q}(\mathcal{X}_2))=0.6385$,    $\lambda_2(\mathcal{Q}(\mathcal{X}_3))=0.2679$, and $\lambda_2(\mathcal{Q}(\mathcal{X}_4))=0.2603$ and the cardinality of the edge set given by $|\mathcal{E}(\mathcal{X}_1)|=12$, $|\mathcal{E}(\mathcal{X}_2)|=13$, $|\mathcal{E}(\mathcal{X}_3)|=10$ and $|\mathcal{E}(\mathcal{X}_4)|=10$. Thus, $\lambda=0.2279$, $\underline{m}=10$ and $\overline{m}=13$.
The consensus protocol is selected as in~\eqref{Eq:ConsensusProtocolA} with $\Omega(\cdot)$ given in Lemma~\ref{Lemma:RhoLarge}, with $p = 0.2$, $q=1.1$, $a=1$ and $b=2$. According to Lemma~\ref{Lemma:RhoLarge}, $\Omega(\cdot)$ is a predefined-time consensus function, satisfying~\eqref{convex_degree} with $\beta(n)=1$ and $\hat{\Omega}(z)=\gamma(a z^{p+1} + b z^{q+1})$. The gain $\kappa_i$ of the consensus protocol is set as $\kappa_i=\kappa$, $i=1,\ldots,n$, where $\kappa$ is given in~\eqref{Eq:KappaProtA} with $T_c=1$. A simulation of the convergence of the consensus algorithm, under the switching dynamic network $\mathcal{X}_{\sigma(t)}$, with nodes' initial conditions given by $x(t_0)=[23.13,   18.33,    8.01,   20.45,    7.57,  -22.77,   12.40,   -9.22,   22.02,  -10.66]^T$ is given in Figure~\ref{Fig:DFD_plot} (top), where the switching signal $\sigma(t)$ is shown in Figure~\ref{Fig:DFD_plot} (bottom). Notice that the consensus state is the average of the nodes' initial conditions, and that convergence is obtained before $T_c$.
\end{example}

\begin{figure}
\begin{center}  
\begin{subfigure}[b]{3cm}
\includegraphics[width=\linewidth]{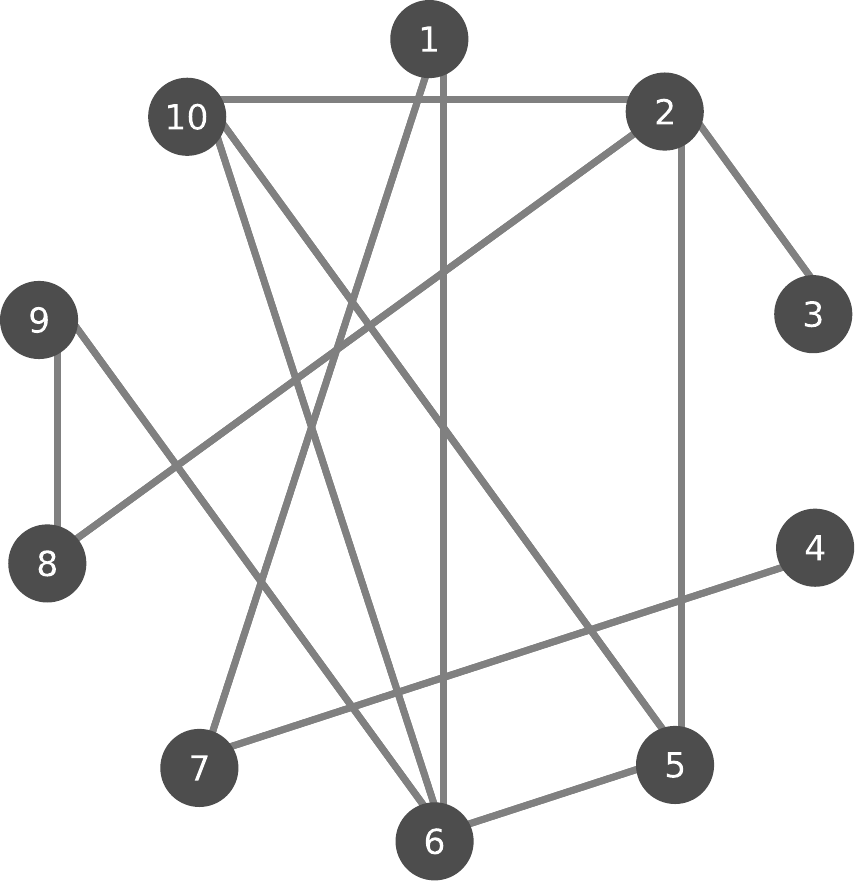}
\caption{$\mathcal{X}_1$}\label{fig:DFD_net1}
\end{subfigure}
\hspace{0.02\textwidth}
\begin{subfigure}[b]{3cm}
\includegraphics[width=\linewidth]{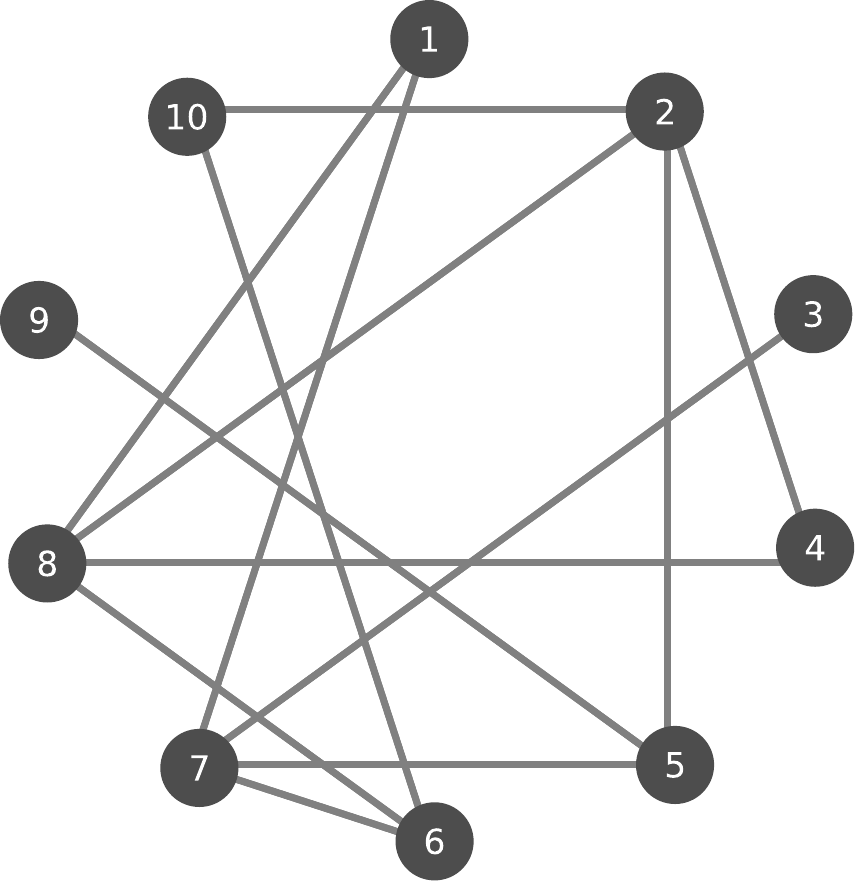}
\caption{$\mathcal{X}_2$}\label{fig:DFD_net2}
\end{subfigure}
\end{center}

\begin{center} 
\begin{subfigure}[b]{3cm}
\includegraphics[width=\linewidth]{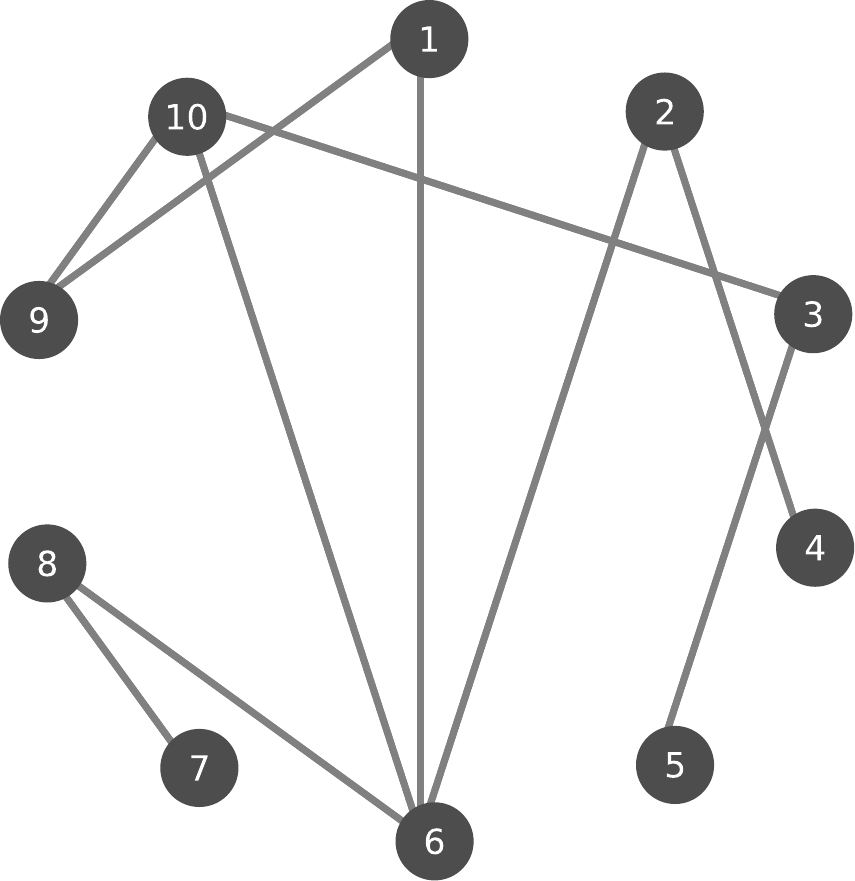}
\caption{$\mathcal{X}_3$}\label{fig:DFD_net3}
\end{subfigure}
\hspace{0.02\textwidth}
\begin{subfigure}[b]{3cm}
\includegraphics[width=\linewidth]{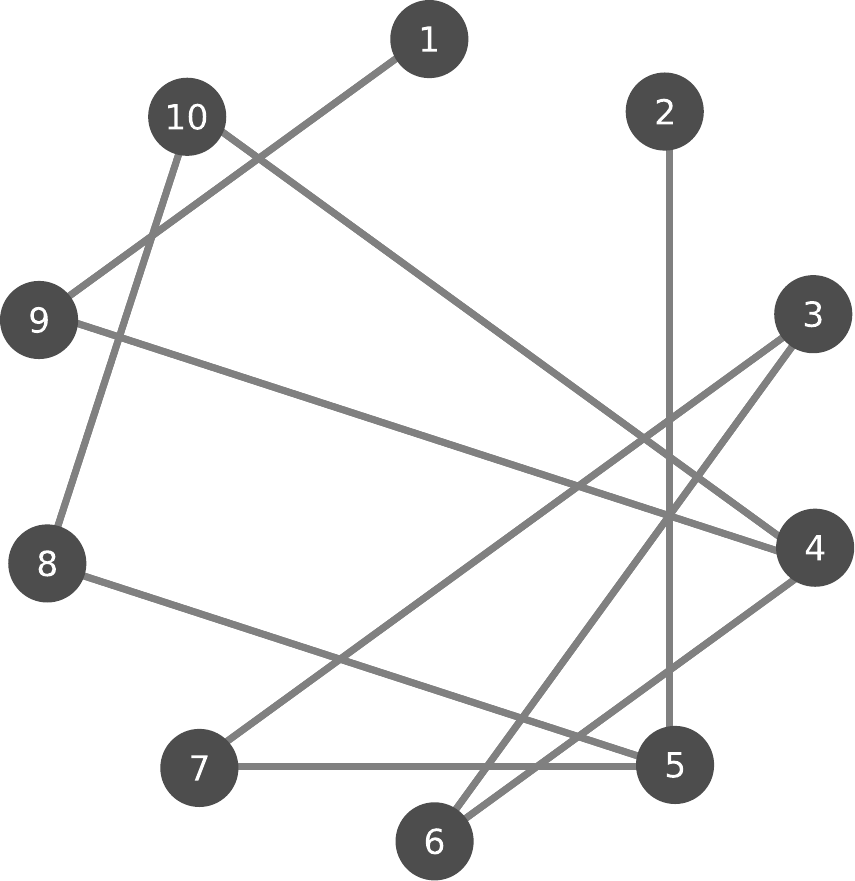}
\caption{$\mathcal{X}_4$}\label{fig:DFD_net4}
\end{subfigure}
\end{center}
\caption{The graphs forming the switching dynamic network of Example~\ref{Ex:Fixed}}
\end{figure}

\begin{figure}[t]
\begin{center}         
\def\svgwidth{12cm}
\begingroup%
  \makeatletter%
  \providecommand\color[2][]{%
    \errmessage{(Inkscape) Color is used for the text in Inkscape, but the package 'color.sty' is not loaded}%
    \renewcommand\color[2][]{}%
  }%
  \providecommand\transparent[1]{%
    \errmessage{(Inkscape) Transparency is used (non-zero) for the text in Inkscape, but the package 'transparent.sty' is not loaded}%
    \renewcommand\transparent[1]{}%
  }%
  \providecommand\rotatebox[2]{#2}%
  \newcommand*\fsize{\dimexpr\f@size pt\relax}%
  \newcommand*\lineheight[1]{\fontsize{\fsize}{#1\fsize}\selectfont}%
  \ifx\svgwidth\undefined%
    \setlength{\unitlength}{420bp}%
    \ifx\svgscale\undefined%
      \relax%
    \else%
      \setlength{\unitlength}{\unitlength * \real{\svgscale}}%
    \fi%
  \else%
    \setlength{\unitlength}{\svgwidth}%
  \fi%
  \global\let\svgwidth\undefined%
  \global\let\svgscale\undefined%
  \makeatother%
  \begin{picture}(1,0.75)%
    \lineheight{1}%
    \setlength\tabcolsep{0pt}%
    \put(0,0){\includegraphics[width=\unitlength]{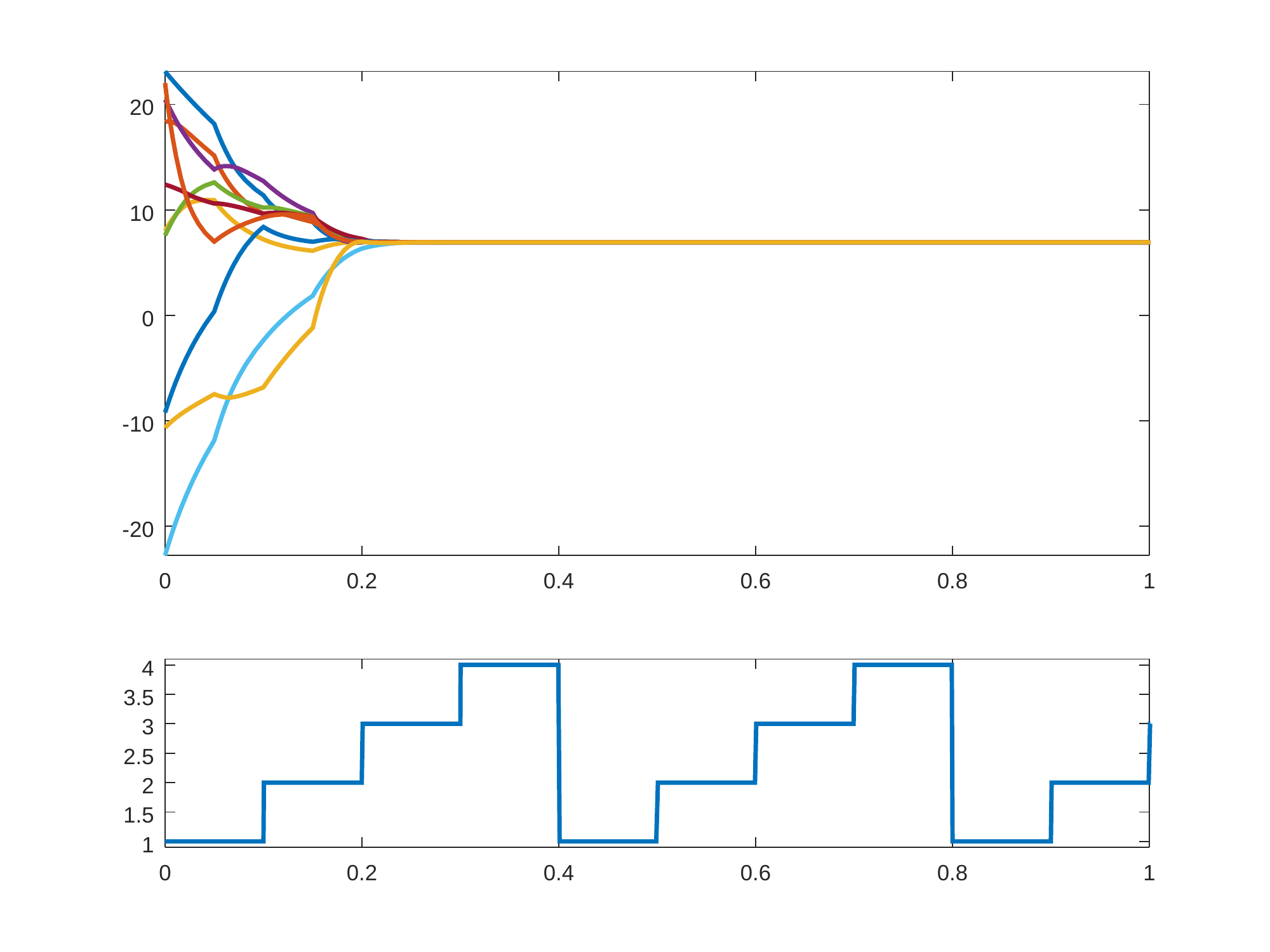}}%
    \put(0.50942738,0.01196306){\color[rgb]{0,0,0}\makebox(0,0)[lt]{\lineheight{1.25}\smash{\begin{tabular}[t]{l}time\end{tabular}}}}%
    \put(0.07595179,0.13017203){\color[rgb]{0,0,0}\rotatebox{90}{\makebox(0,0)[lt]{\lineheight{1.25}\smash{\begin{tabular}[t]{l}$\sigma(t)$\end{tabular}}}}}%
    \put(0.07595179,0.44445774){\color[rgb]{0,0,0}\rotatebox{90}{\makebox(0,0)[lt]{\lineheight{1.25}\smash{\begin{tabular}[t]{l}$x_i(t)$\end{tabular}}}}}%
  \end{picture}%
\endgroup%
\end{center}
\caption{Convergence of the predefined-time consensus algorithm for Example \ref{Example1} under dynamic networks with $T_c$}\label{Fig:DFD_plot}
\end{figure} 

\subsection{A simpler predefined-time consensus algorithm for static networks}

In this subsection we will analyze the consensus protocol proposed in~\eqref{Eq:ConsensusProtocolB}. We first show that~\eqref{Eq:ConsensusProtocolB} is a consensus protocol with fixed-time convergence for static networks and we derive the conditions under which predefined-time convergence bounded is obtained. Afterwards, we show that for dynamic networks,~\eqref{Eq:ConsensusProtocolB} is fixed-time convergence. 

\begin{remark}
Unlike protocol~\eqref{Eq:ConsensusProtocolB} which requires, at each time instant, one evaluation of the nonlinear predefined-time consensus function (hereinafter introduced) per neighbors, the second protocol only requires a single evaluation and ensures predefined-time consensus for static topologies and fixed-time convergence for dynamic networks.
\label{Remark:Simpler}
\end{remark}

\begin{theorem}
\label{Th:ConsensusB}
Let $\mathcal{X}$ be the connectivity graph for the static network, and let $\Omega(\cdot)$ be a predefined-time consensus function with associated $d$, $\beta(\cdot)$  and $\hat{\Omega}(\cdot)$ such that~\eqref{convex_degree} holds. Then, if $\mathcal{X}$ is a connected graph and $\kappa_i>0$, $i=1,\ldots, n$, then ~\eqref{Eq:ConsensusProtocolB} is a consensus algorithm with fixed-time convergence. Moreover, if $\mathcal{X}$ is a connected graph and $\kappa_i\geq\frac{1}{\lambda_2(\mathcal{Q})\beta(n)^{2-d}T_c}$, $i=1,\ldots, n$, then ~\eqref{Eq:ConsensusProtocolB} is a consensus algorithm with predefined-time convergence bounded by $T_c$, i.e. all trajectories of~\eqref{ConsensusDynamicB} converges to the consensus state $x_1=\cdots=x_n$ in predefined-time bounded by $T_c$. 
\end{theorem}
\begin{proof}
First notice that the dynamic of the network under the consensus algorithm \eqref{Eq:ConsensusProtocolB} is given by
\begin{equation}
\label{ConsensusDynamicB}
\dot{x}=-\mathcal{F}(Q(\mathcal{X}_{\sigma(t)})x).
\end{equation} where $\mathcal{F}(\cdot)$ is defined as in~\eqref{Eq:Fp}. Thus, the equilibrium subspace is given by $\mathcal{Z}(x)=\{x:x_1=\cdots=x_n\}$, i.e. at the equilibrium, consensus is achieved.

Consider the radially unbounded Lyapunov function candidate
\begin{equation}
V(x) = \sqrt{\lambda_2(\mathcal{Q})}\beta(n)\sqrt{x^T \mathcal{Q} x},
\label{Eq:LyapunovStatic}
\end{equation}
which satisfies that $V(x)=0$ if and only if $x\in\mathcal{Z}(x)$, and whose time-derivative along the trajectory of system~\eqref{ConsensusDynamicB} yields
\begin{align}
\dot{V}(x) = -\frac{\sqrt{\lambda_2(\mathcal{Q})} \beta(n)}{\sqrt{x^T \mathcal{Q} x}}x^T\mathcal{Q}\mathcal{F}(\mathcal{Q}x) 
= -\lambda_2(\mathcal{Q})\beta(n)^2V^{-1}e^T \mathcal{F}(e) 
= -\lambda_2(\mathcal{Q})\beta(n)^{2-d}V^{-1}\sum_{i=1}^n\beta(n)^d\kappa_i\Omega(|e_i|).
\label{Eq:LyapStatDot1}
\end{align}
where $e = \mathcal{Q}x$. Using the fact that $\Omega(\cdot)$ is a predefined time consensus function, then it follows that
\begin{align*}
\sum_{i=1}^n\beta(n)^d\kappa_i\Omega\left(|e_i|\right) \geq \kappa\sum_{i=1}^n\beta(n)^d\Omega\left(|e_i|\right) \geq
\kappa\hat{\Omega}\left(\beta(n)\|e\|_2\right),
\end{align*}
where $\kappa = \min\{\kappa_1,\dots,\kappa_n\}$.
Moreover, it follows from Lemma~\ref{Lemma:Hardy} and Lemma \ref{lemma:Lambda2} that
\begin{equation*}
\|e\|_2 = \sqrt{e^Te} = \sqrt{x^T \mathcal{Q}^2x} \geq \sqrt{\lambda_2(\mathcal{Q})}\sqrt{x^T \mathcal{Q}x}=\beta(n)^{-1}V.
\end{equation*}
Therefore:
\begin{align}
\label{prot1_res12}
\kappa\hat{\Omega}\left(\beta(n)\|e\|_2\right) \geq 
\kappa\hat{\Omega}\left(V\right) = \kappa V\Psi(V).
\end{align}

Hence, it follows from~\eqref{Eq:LyapStatDot1}, and \eqref{prot1_res12} that
\begin{equation}
\dot{V}(x) \leq -\kappa\lambda_2(\mathcal{Q})\beta(n)^{2-d}\Psi(V).
\end{equation}
and, according to Theorem~\ref{thm:weak_pt}, protocol~\eqref{Eq:ConsensusProtocolB}
guarantees that $x$ converges to $\mathcal{Z}(x)=\{x:x_1=\cdots=x_n\}$ in a fixed-time bounded by
$$
\sup_{x_0 \in \mathbb{R}^n} T(x_0)\leq\frac{1}{\kappa\lambda_2(\mathcal{Q})\beta(n)^{2-d}}
$$
Thus, if
\begin{equation}
\label{Eq:GainConsB}
\kappa_i\geq\kappa=\frac{1}{\lambda_2(\mathcal{Q})\beta(n)^{2-d}T_c}
\end{equation}
then
\begin{equation}
\dot{V}(x) \leq -\frac{1}{T_c}\Psi(V),
\end{equation}
and according to Theorem~\ref{thm:weak_pt}, protocol~\eqref{Eq:ConsensusProtocolB}
guarantees that the consensus is achieved before a predefined-time $T_c$. 
\end{proof}

\begin{example}
\label{Example2}
Consider a network $\mathcal{X}$ composed of 10 agents with communication topology as shown in Figure~\ref{Fig:FDD_net}, which has an algebraic connectivity of $\lambda_2(\mathcal{Q}) = 0.27935$. 
Let the initial condition be $x(t_0)=[3.65,-8.99,-3.26,-0.03, 4.52,13.53,15.85,-0.53,-9.97,-13.91]^T$. Then, the convergence of  algorithm~\eqref{Eq:ConsensusProtocolB} under the graph topology $\mathcal{X}$ using $\Omega(z)=\frac{1}{p}\exp(z^p)z^{2-p}$, $T_c = 1$ and $p=0.5$ is shown in Figure~\ref{Fig:FDD_plot} where $\kappa_i$ is selected as $\kappa_i=\kappa$, $i=1,\ldots,n$ with $\kappa$ as in~\eqref{Eq:GainConsB} with $\beta(n)=\frac{1}{n}$.
\end{example}

\begin{figure}[t]
\begin{center}         
\includegraphics[width=5cm]{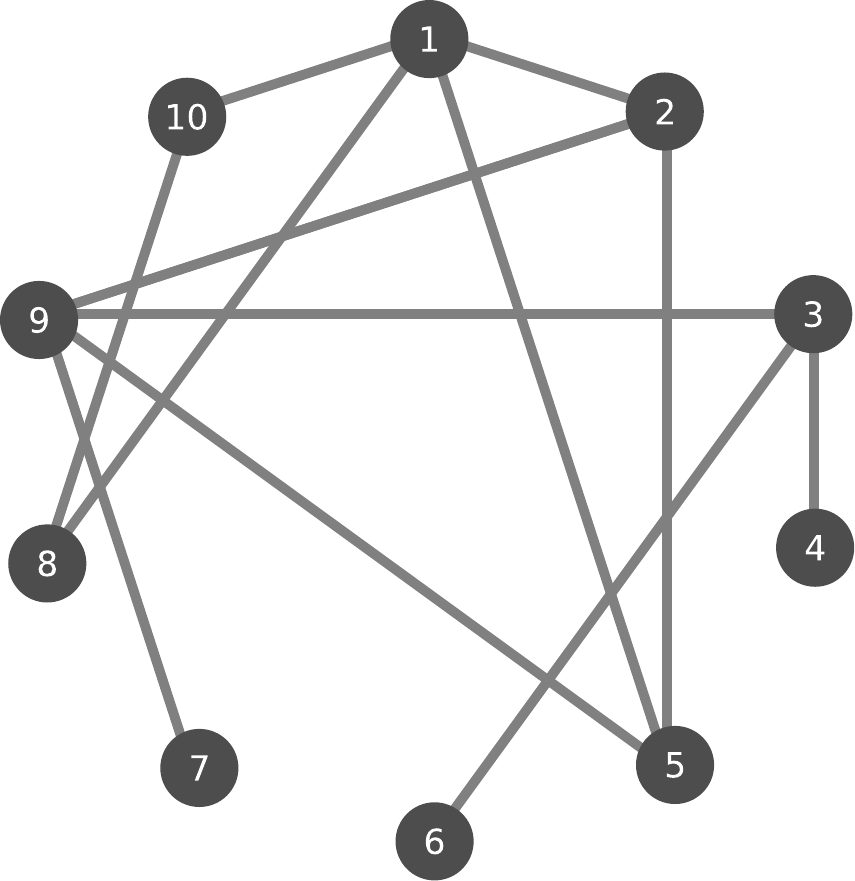}
\end{center}
\caption{Network $\mathcal{X}$ used for Example \ref{Example2}}\label{Fig:FDD_net}
\end{figure}

\begin{figure}[t]
\def\svgwidth{12cm}
\begin{center}         
\begingroup%
  \makeatletter%
  \providecommand\color[2][]{%
    \errmessage{(Inkscape) Color is used for the text in Inkscape, but the package 'color.sty' is not loaded}%
    \renewcommand\color[2][]{}%
  }%
  \providecommand\transparent[1]{%
    \errmessage{(Inkscape) Transparency is used (non-zero) for the text in Inkscape, but the package 'transparent.sty' is not loaded}%
    \renewcommand\transparent[1]{}%
  }%
  \providecommand\rotatebox[2]{#2}%
  \newcommand*\fsize{\dimexpr\f@size pt\relax}%
  \newcommand*\lineheight[1]{\fontsize{\fsize}{#1\fsize}\selectfont}%
  \ifx\svgwidth\undefined%
    \setlength{\unitlength}{420bp}%
    \ifx\svgscale\undefined%
      \relax%
    \else%
      \setlength{\unitlength}{\unitlength * \real{\svgscale}}%
    \fi%
  \else%
    \setlength{\unitlength}{\svgwidth}%
  \fi%
  \global\let\svgwidth\undefined%
  \global\let\svgscale\undefined%
  \makeatother%
  \begin{picture}(1,0.75)%
    \lineheight{1}%
    \setlength\tabcolsep{0pt}%
    \put(0,0){\includegraphics[width=\unitlength]{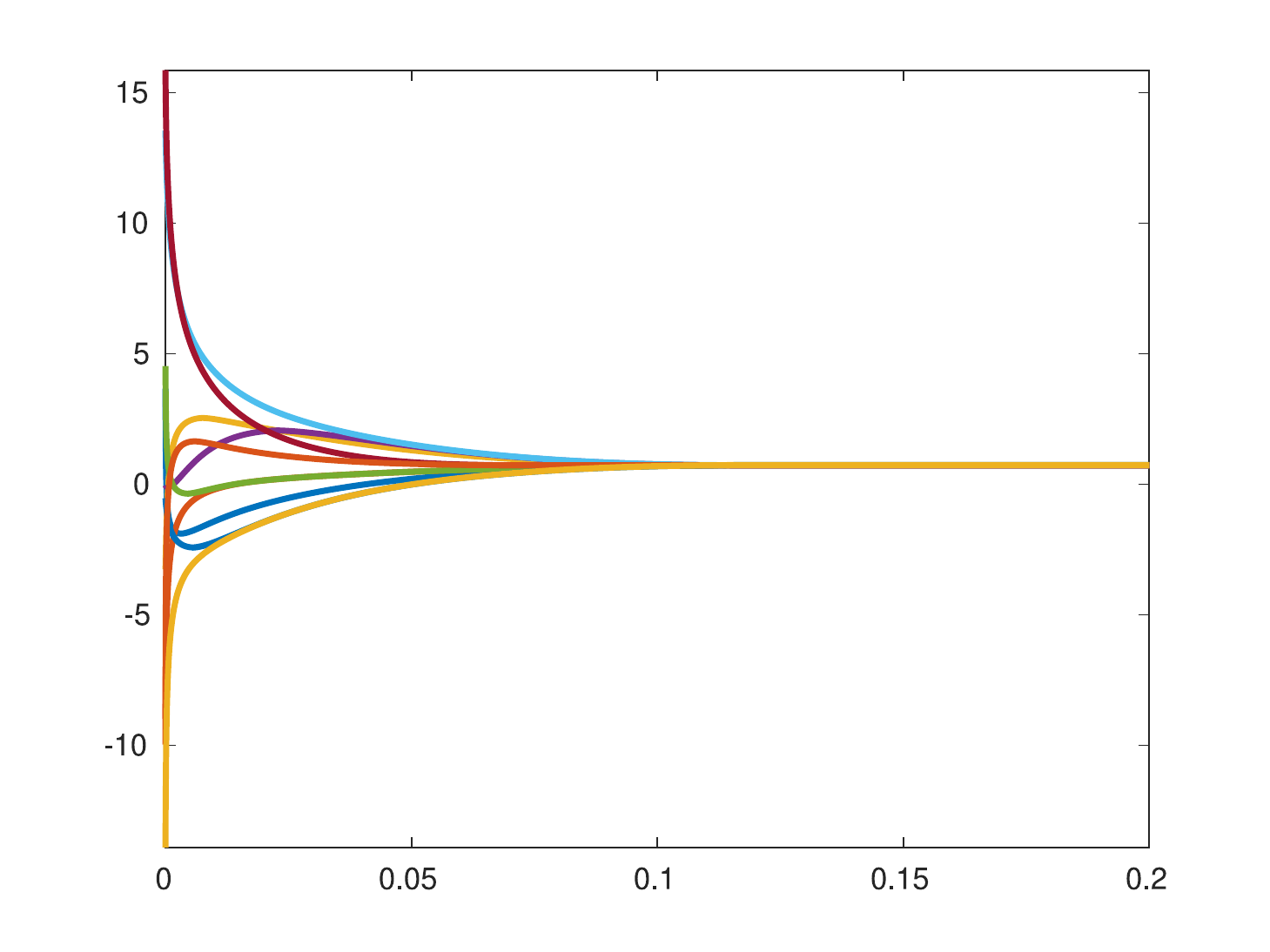}}%
    \put(0.49054779,0.01822734){\color[rgb]{0,0,0}\makebox(0,0)[lt]{\lineheight{1.25}\smash{\begin{tabular}[t]{l}time\end{tabular}}}}%
    \put(0.07238055,0.32870163){\color[rgb]{0,0,0}\rotatebox{90}{\makebox(0,0)[lt]{\lineheight{1.25}\smash{\begin{tabular}[t]{l}$x_i(t)$\end{tabular}}}}}%
  \end{picture}%
\endgroup%
\end{center}
\caption{Convergence of the consensus algorithm for Example \ref{Example2} with $T_c=1$}
\label{Fig:FDD_plot}
\end{figure}

\begin{remark}
Notice that the convergence-time to the consensus state is a function of the algebraic connectivity of the network~\cite{Olfati-Saber2007,Ning2017b}. Hence, to compute the gain $\kappa$ to obtain predefined-time convergence, we assume knowledge of a lower-bound of the algebraic connectivity of the network. For static networks, there exist several algorithms for distributively estimating the algebraic connectivity~\cite{Li2013b,Aragues2014,Montijano2017}. For instance, the algorithm proposed in~\cite{Aragues2014} provides an asymptotic estimation, which is always a lower bound of the true algebraic connectivity. Another scenario is, given an estimate of the size of the network~\cite{Shames2012,You2017} to consider a worst-case $\lambda_2$ from information specific to the problem. 
\end{remark}

\begin{remark}
\label{Remark:Sym}
We have shown in Theorem~\ref{Th:ConsensusB}, using the Lyapunov function~\eqref{Eq:LyapunovStatic}, that~\eqref{Eq:ConsensusProtocolB} is a predefined-time consensus algorithm for static networks. However, since the Lyapunov function~\eqref{Eq:LyapunovStatic} is a function of the Laplacian matrix of the graph, then, the predefined-time convergence for switching dynamic networks cannot be justified as in the proof of Theorem~\ref{Th:ConsensusA}. To show that (at least) fixed-time stability is maintained under an arbitrary switching signal, non-smooth Lyapunov analysis~\cite{Bacciotti2006} is used in the following theorem. Notice that in~\cite{Zuo2014,Ning2017b}, the consensus protocol~\eqref{Eq:ConsensusProtocolB} using the predefined-time consensus function given in Lemma~\ref{Lemma:RhoLarge} was proposed, but restricted only to the case where $p=1-s$ and $q=1+s$ with $0<s<1$, which was justified only as a fixed-time consensus protocol for static networks.
\end{remark}

\begin{theorem}
\label{Th:ConsensusB2}
If $\kappa_i>0$, $i=1,\ldots,n$, then~\eqref{Eq:ConsensusProtocolB} is a consensus algorithm, with fixed-time convergence, for dynamic networks arbitrarily switching among connected graphs.
\end{theorem}
\begin{proof}
Let $\mathcal{X}_\sigma(t)$ be a switching dynamic network,  and consider the (Lipschitz continuous) Lyapunov function candidate
\begin{equation}
    V(x)=\max(x_1,\cdots,x_n)-\min(x_1,\cdots,x_n),
\label{Eq:LyapMaxMin}
\end{equation}
which is differentiable almost everywhere and positive definite. Notice that $V(x)=0$ if and only if $x\in\mathcal{Z}(x)=\{x:x_1=\cdots=x_n\}$. Let $\mathcal{X}_l$ be the current graph topology, then, if $x_j=\max(x_1,\cdots,x_n)$ and $x_i=\min(x_1,\cdots,x_n)$ for a nonzero interval
\begin{equation}
\label{Eq:LyapMaxMinA}
\mathcal{D}^{+}V(x)=\kappa_je_j^{-1}\Omega(|e_j|)-\kappa_ie_i^{-1}\Omega(|e_i|),
\end{equation}
where $\mathcal{D}^{+}V(x)$ is the Dini derivative, and $e_i$ and $e_j$ are as in~\eqref{Eq:ConsensusProtocolB}. However, since $x_j\geq x_k$, $\forall k\in\mathcal{V}(\mathcal{X})$, it follows that $\sign(e_j)=-1$ whenever $e_j\neq 0$, and thus $e_j^{-1}\Omega(|e_j|)\leq0$. By a similar argument it follows that $e_i^{-1}\Omega(|e_i|)\geq 0$. Thus, $\dot{V}(x)\leq 0$.

Notice that the largest invariant set such that $\mathcal{D}^{+}V(x)=0$ is  $\{x:\max(x_1,\cdots,x_n)=\min(x_1,\cdots,x_n)\}$, because otherwise, since the graph is connected, there is a path from $j$ to $i$, such that there exists a node $k$ that belongs to such path, such that $x_j=x_k=\max(x_1,\cdots,x_n)$ but $x_k\neq x_{k'}$ for some $k'\in\mathcal{N}_k(\mathcal{X}_l)$. Thus, $e_k\neq0$ and in turn makes $e_j<0$ and, therefore, $\mathcal{D}^{+}V(x)=0$ does not hold.
Thus, using LaSalle's invariance principle~\cite{Khalil2002},~ \eqref{ConsensusDynamicB} converges asymptotically to $\{x:\max(x_1,\cdots,x_n)=\min(x_1,\cdots,x_n)\}$, which implies that $x_1=\cdots=x_n$, i.e. consensus is achieved asymptotically. 

Now, consider a switched dynamic network composed of connected graphs and driven by the arbitrary switching signal $\sigma(t)$. Since~\eqref{Eq:LyapMaxMinA} is a common Lyapunov function for the evolution of the system under each graph $\mathcal{X}_l$, the asymptotic convergence of the system is preserved in a dynamic network under arbitrary switching~\cite{Liberzon2003}.

Finally, since we have shown in the proof of Theorem~\ref{Th:ConsensusB} that, if the topology is static and connected, ~\eqref{Eq:LyapunovStatic} goes to zero in a fixed-time bounded by the constant $\frac{1}{\kappa\lambda_2(\mathcal{Q})\beta(n)^{2-d}}$ is the active topology, then, under this scenario, ~\eqref{Eq:LyapMaxMinA} also goes to zero in a fixed-time bounded by the same constant,  because if $x$ is such that~\eqref{Eq:LyapunovStatic} is zero then also ~\eqref{Eq:LyapMaxMinA} is zero. 
Since, in the switching case,~\eqref{Eq:LyapMaxMinA} is still decreasing and continuous, then it follows that~\eqref{Eq:LyapMaxMinA} goes to zero in a fixed-time lower or equal than the lowest time such that there exists a connected topology $\mathcal{X}_l$, such that the sum of time intervals in which $\mathcal{X}_l$ has been active is greater than $\frac{1}{\kappa\lambda_2(\mathcal{Q})\beta(n)^{2-d}}$. Since this upper bound is independent of the initial state of the agents, then fixed-time convergence is obtained under switching topologies.
\end{proof}

\begin{example}
\label{Ex:Fixed}
Consider a switching dynamic network $\mathcal{X}_{\sigma(t)}$ with $\mathcal{F}=\{\mathcal{X}_1,\mathcal{X}_2,\mathcal{X}_3,\mathcal{X}_4\}$ with graphs $\mathcal{X}_i$, $i=1,\ldots,4$ shown in Figure~\ref{fig:FDD_net1}-\ref{fig:FDD_net4}. Figure~\ref{Fig:FDD_plot_sw} (top) show the evolution of the agents' state, with the consensus protocol~\eqref{Eq:ConsensusProtocolB}, under the switching dynamic network $\mathcal{X}_{\sigma(t)}$, with switching signal $\sigma(t)$ shown in Figure~\ref{Fig:FDD_plot_sw} (bottom) where $\kappa$ is selected as $\kappa=$.

\begin{figure}
\begin{center}  
\begin{subfigure}[b]{3cm}
\includegraphics[width=\linewidth]{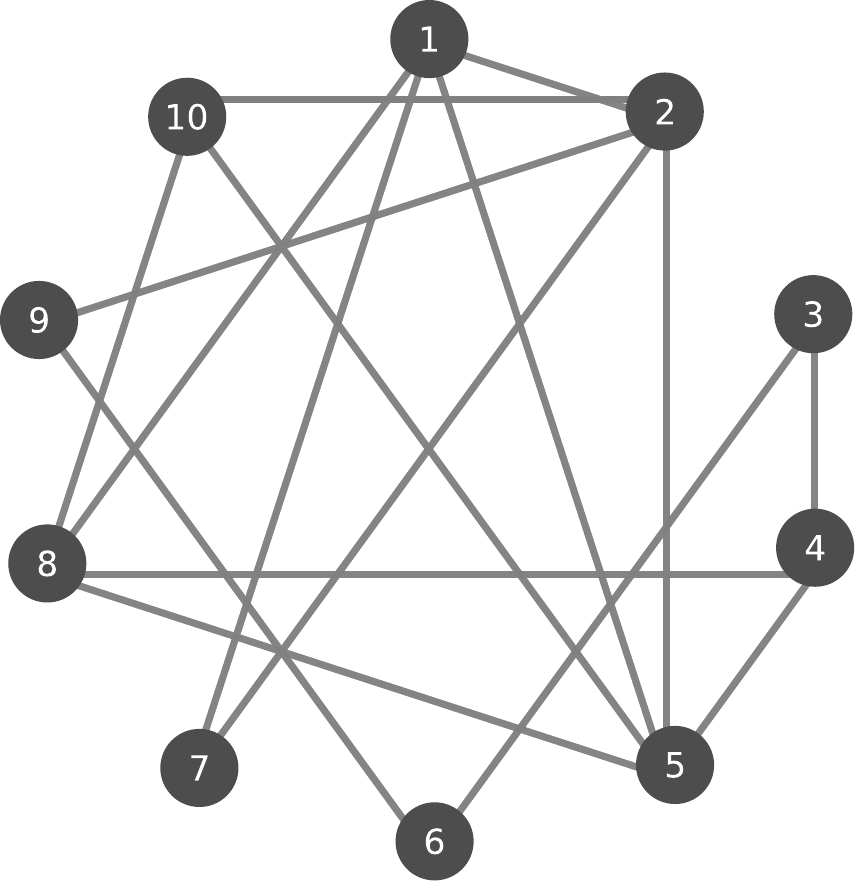}
\caption{$\mathcal{X}_1$}\label{fig:FDD_net1}
\end{subfigure}
\hspace{0.02\textwidth}
\begin{subfigure}[b]{3cm}
\includegraphics[width=\linewidth]{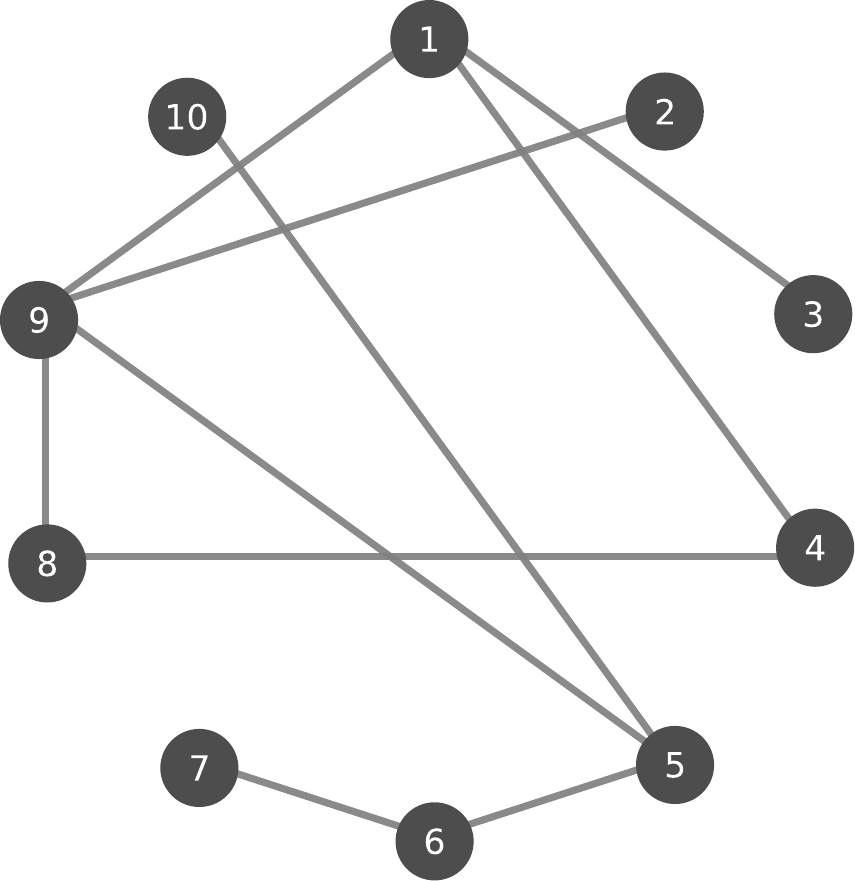}
\caption{$\mathcal{X}_2$}\label{fig:FDD_net2}
\end{subfigure}
\end{center}

\begin{center} 
\begin{subfigure}[b]{3cm}
\includegraphics[width=\linewidth]{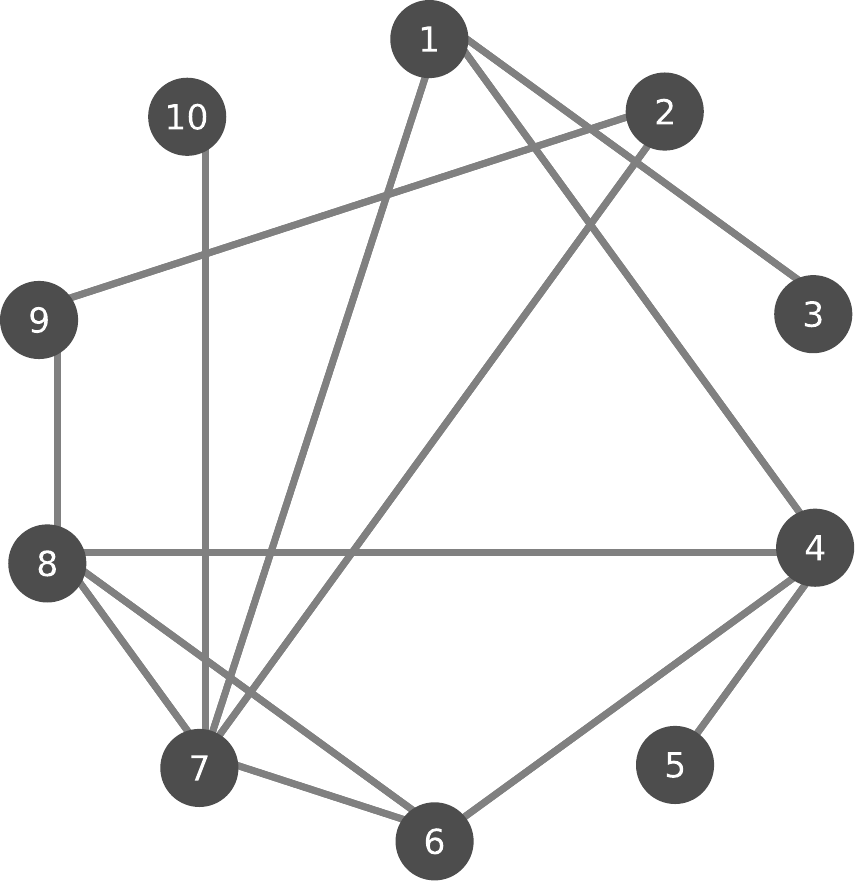}
\caption{$\mathcal{X}_3$}\label{fig:FDD_net3}
\end{subfigure}
\hspace{0.02\textwidth}
\begin{subfigure}[b]{3cm}
\includegraphics[width=\linewidth]{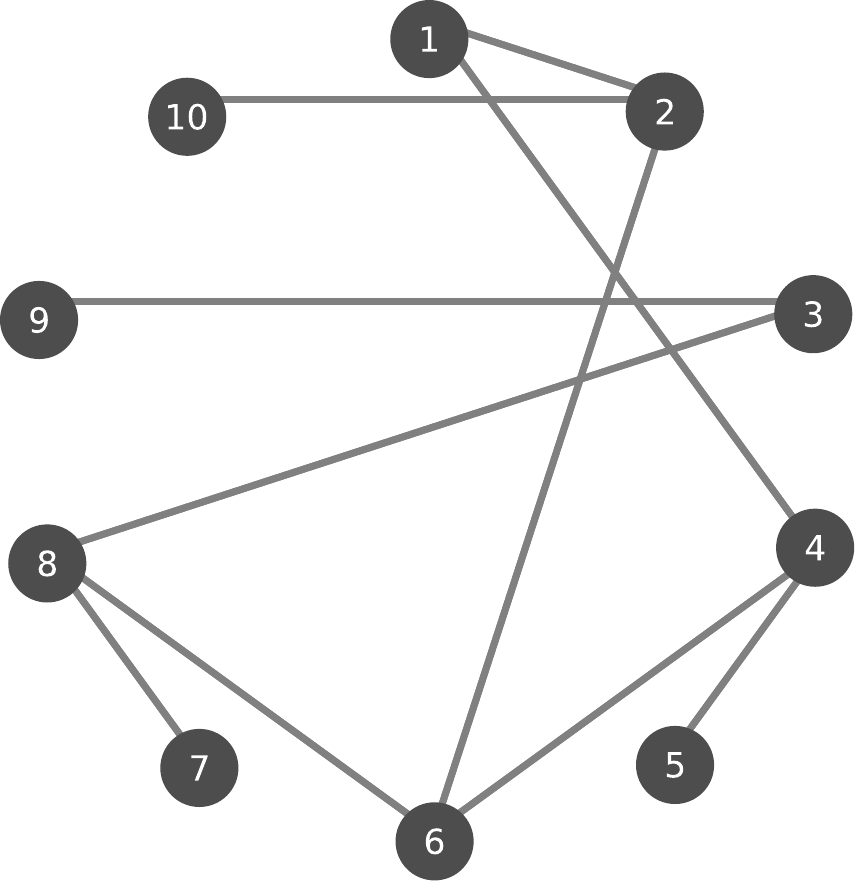}
\caption{$\mathcal{X}_4$}\label{fig:FDD_net4}
\end{subfigure}
\end{center}
\caption{The graphs forming the switching dynamic network of Example~\ref{Ex:Fixed}}
\end{figure}

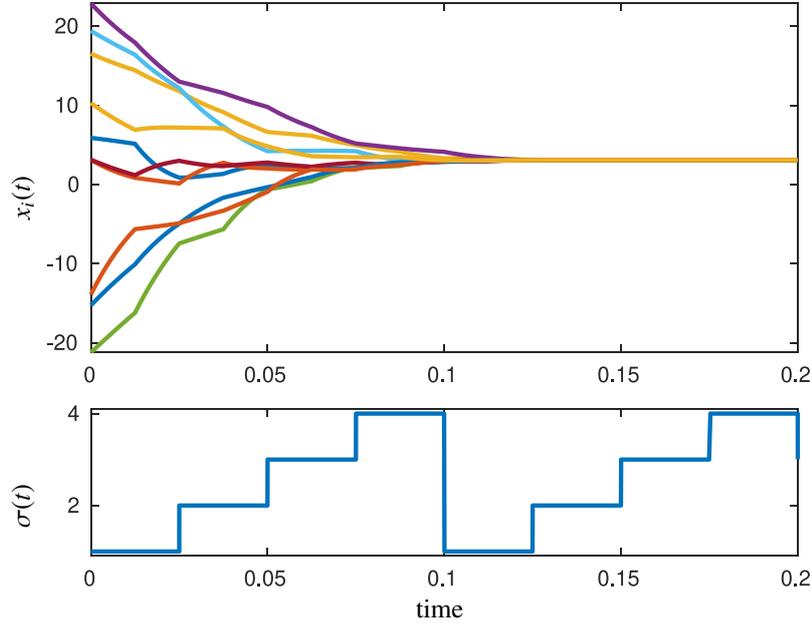
\begin{figure}[t]
\begin{center}         
\def\svgwidth{12cm}
\begingroup%
  \makeatletter%
  \providecommand\color[2][]{%
    \errmessage{(Inkscape) Color is used for the text in Inkscape, but the package 'color.sty' is not loaded}%
    \renewcommand\color[2][]{}%
  }%
  \providecommand\transparent[1]{%
    \errmessage{(Inkscape) Transparency is used (non-zero) for the text in Inkscape, but the package 'transparent.sty' is not loaded}%
    \renewcommand\transparent[1]{}%
  }%
  \providecommand\rotatebox[2]{#2}%
  \newcommand*\fsize{\dimexpr\f@size pt\relax}%
  \newcommand*\lineheight[1]{\fontsize{\fsize}{#1\fsize}\selectfont}%
  \ifx\svgwidth\undefined%
    \setlength{\unitlength}{420bp}%
    \ifx\svgscale\undefined%
      \relax%
    \else%
      \setlength{\unitlength}{\unitlength * \real{\svgscale}}%
    \fi%
  \else%
    \setlength{\unitlength}{\svgwidth}%
  \fi%
  \global\let\svgwidth\undefined%
  \global\let\svgscale\undefined%
  \makeatother%
  \begin{picture}(1,0.75)%
    \lineheight{1}%
    \setlength\tabcolsep{0pt}%
    \put(0,0){\includegraphics[width=\unitlength,page=1]{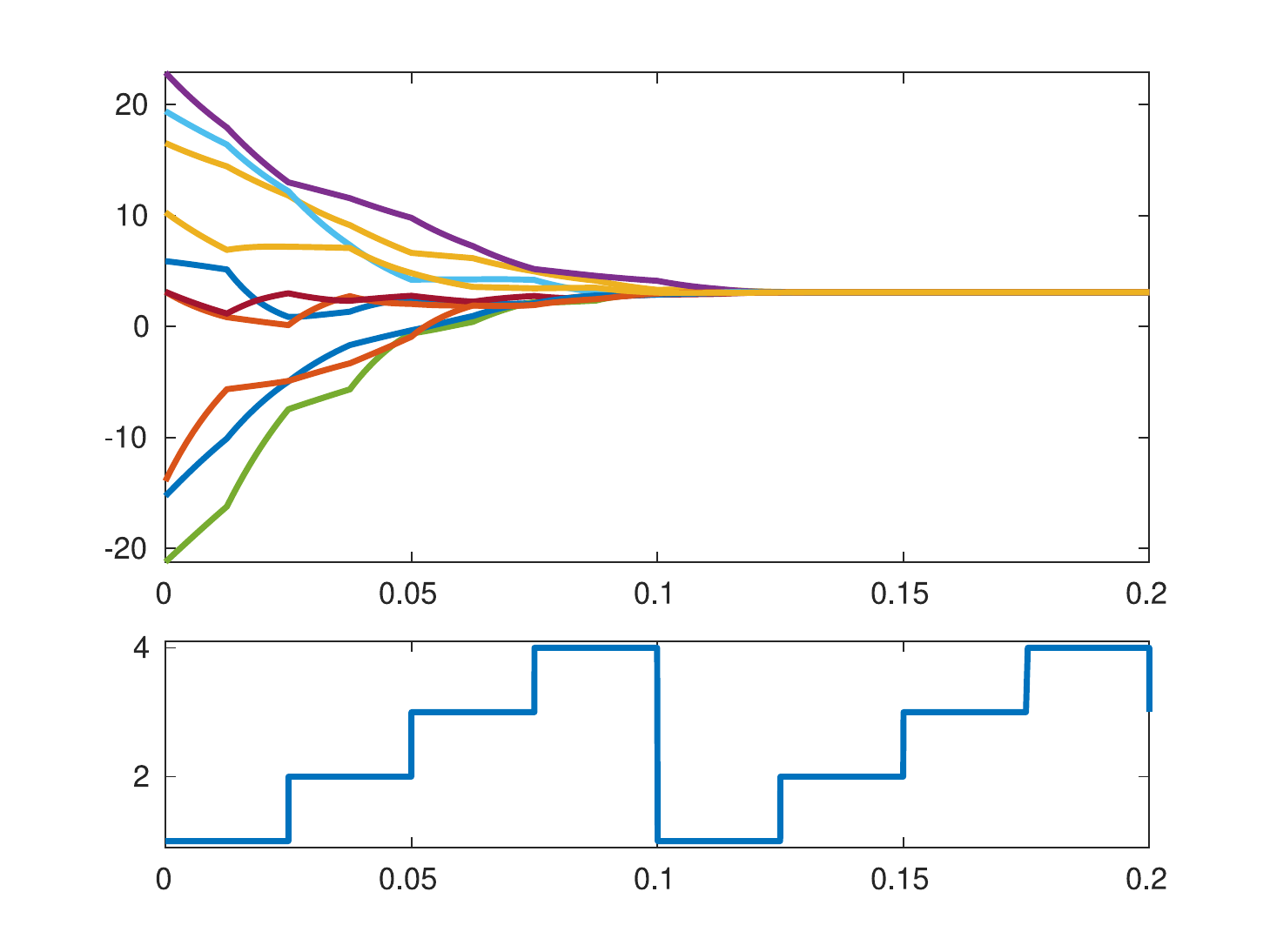}}%
    \put(0.06328819,0.45308234){\color[rgb]{0,0,0}\rotatebox{90}{\makebox(0,0)[lt]{\lineheight{1.25}\smash{\begin{tabular}[t]{l}$x_i(t)$\end{tabular}}}}}%
    \put(0.06328819,0.11355161){\color[rgb]{0,0,0}\rotatebox{90}{\makebox(0,0)[lt]{\lineheight{1.25}\smash{\begin{tabular}[t]{l}$\sigma(t)$\end{tabular}}}}}%
    \put(0.48694806,0.01448054){\color[rgb]{0,0,0}\makebox(0,0)[lt]{\lineheight{1.25}\smash{\begin{tabular}[t]{l}time\end{tabular}}}}%
  \end{picture}%
\endgroup%
\end{center}
\caption{Convergence of the predefined-time consensus algorithm for Example~\ref{Example1} under dynamic networks with $T_c$}\label{Fig:FDD_plot_sw}
\end{figure}

\end{example}

\section{Application to predefined-time multi-agent formation control}
\label{Sec:Formation}
In this section, it is described how the proposed method can be applied to achieve a distributed formation with predefined-time convergence in a multi-agent system, where agents only have information on the relative displacement of their neighbors.

Let $z_i$ be the $i$-th agent position and $d_{ji}^*$ be the displacement requirement between the $i-$th and the $j-$th agent in the desired formation. A displacement requirement $d_{ji}^*$ for all $i,j\in\mathcal{V}$ is said to be feasible if there exists a position $z^*$ such that $\forall i,j\in\mathcal{V}$, $z_j^*-z_i^*=d_{ji}^*$ where $z_i^*$ and $z_j^*$ are the $i$-th and $j$-th element of $z^*$, respectively. The aim of the multi-agent formation control problem is to guarantee that each agent converges to a position $z$ where the displacement requirement is fulfilled. 

\begin{corollary}
Let $z$ represent the position of the agents and let $z^*=[z_1^*\  \cdots \ z_n^*]^T$ be a feasible displacement requirement for a desired formation. The following position update rules 
\begin{enumerate}[label=(\roman*)]
    \item 
\begin{equation}
\label{Eq:FormConsensusProtocolA} 
\dot{z}_i=\kappa\sum_{j\in\mathcal{N}_i(\mathcal{X}_{\sigma(t)})}\sqrt{a_{ij}}\hat{e}_{ij}^{-1}\Omega(\hat{e}_{ij}), \ \ \ \ \ \hat{e}_{ij}=\sqrt{a_{ij}}(z_j(t)-z_i(t)-d_{ji}^*),
\end{equation}
\item
\begin{equation}
\label{Eq:FormConsensusProtocolB}
\dot{z}_i=\kappa\hat{e}_i^{-1}\Omega(\hat{e}_i), \ \ \ \ \ \hat{e}_i=\sum_{j\in\mathcal{N}_i(\mathcal{X}_{\sigma(t)})}a_{ij}(z_j(t)-z_i(t)-d_{ji}^*).
\end{equation}
\end{enumerate}
solve the displacement formation control in predefined-time if $k$ is selected as in Theorem~\ref{Th:ConsensusA} for~\eqref{Eq:FormConsensusProtocolA} and as in Theorem~\ref{Th:ConsensusB} for~\eqref{Eq:FormConsensusProtocolB}.
\end{corollary}

\begin{proof}
The proof follows by noticing that the dynamic of $x=z-z^*$ for $\dot{z}_i$ given in~\eqref{Eq:FormConsensusProtocolA}, coincides with the dynamics of~\eqref{Eq:ConsensusProtocolA} and the dynamic of $x=z-z^*$ for $\dot{z}_i$ given in~\eqref{Eq:FormConsensusProtocolB}, coincides with the dynamics of~\eqref{Eq:ConsensusProtocolB}. Thus, the displacement formation control is a consensus problem over the variable $x=z-z^*$. Thus, $x$ converges to $x=\alpha\mathbf{1}$ in a predefined-time $T(x_0)\leq T_c$, where $\alpha$ is a constant value. Therefore, $z$ converges to $z=z^*+\alpha\mathbf{1}$. Notice that $z$ satisfies the displacement requirement, since $z_j-z_i=z_j^*-z_i^*$ is satisfied.
\end{proof}

\begin{example}
\label{ExampleFormation}

Consider a system composed of $20$ agents placed with uniformly distributed random initial conditions over $[1,3]^2$ in the $xy-$plane, shown with red dots in Figure~\ref{Fig:formationTraj}. The displacement conditions $d_{ij}^*$ are given such that the agents achieve a formation as given by the blue dots in Figure~\ref{Fig:formationRef}. Two agents are connected if the distance between them is less or equal than a communication range of 0.5m. Notice that, as the agents move the connectivity graph changes. The formation control for each agent is designed as in~\eqref{Eq:FormConsensusProtocolA}, with predefined-time bound $T_c=1$ and $\lambda_2$ for the less connected case, which is a line graph. The convergence of the agents towards the formation is shown in green in Figure~\ref{Fig:formationTraj}. Notice that the agents converge to a formation where the displacement condition in Figure~\ref{Fig:formationRef} is satisfied but where a global position for the nodes is not predetermined.
\end{example}

\begin{figure}[t]
\begin{center}
\begin{subfigure}{7cm}
\includegraphics[width=7cm]{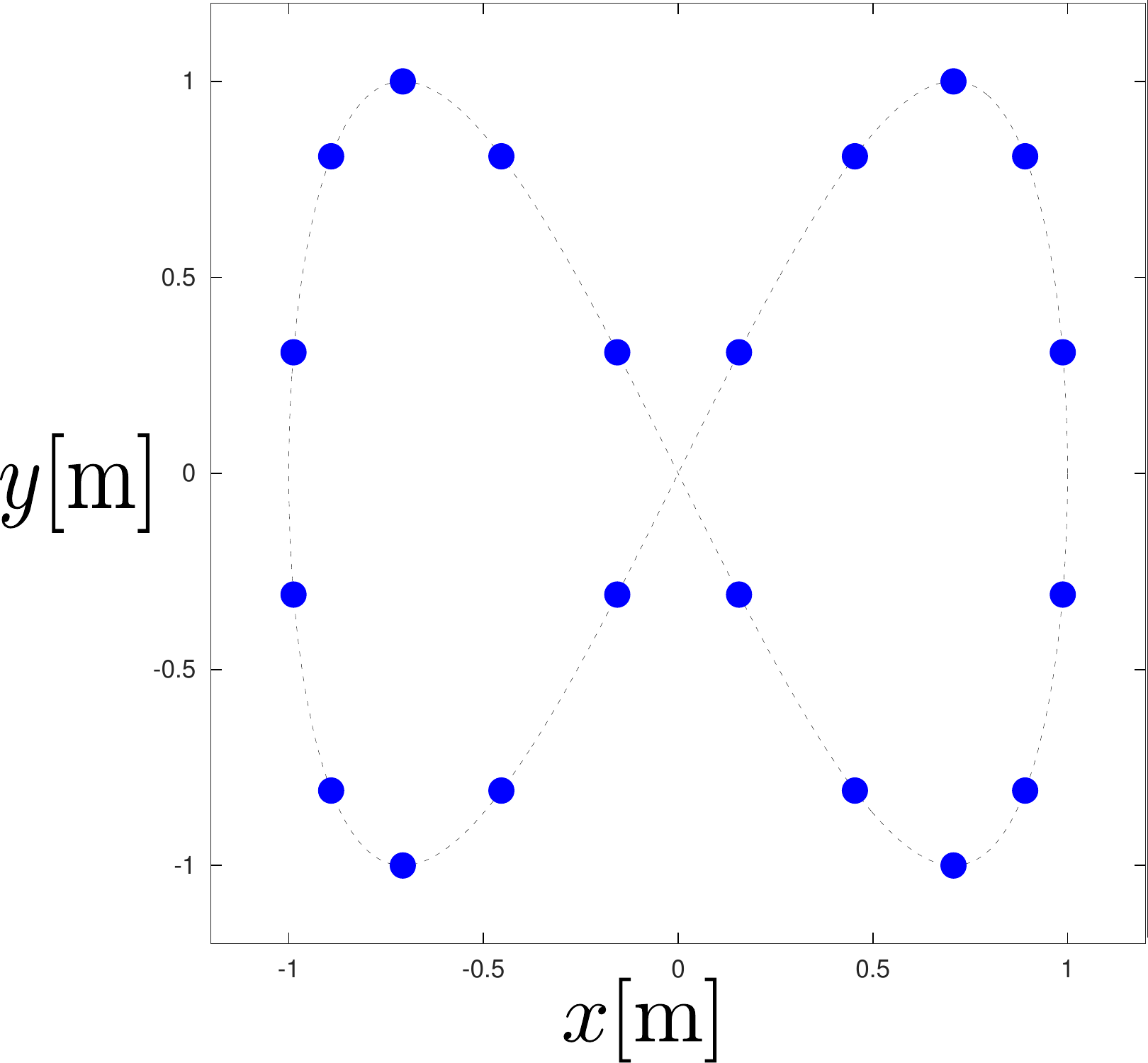}
\caption{A formation satisfying the displacement requirement for the Example \ref{ExampleFormation}}\label{Fig:formationRef}
\end{subfigure}\hspace{2cm}
\begin{subfigure}{7cm}
\includegraphics[width=7cm]{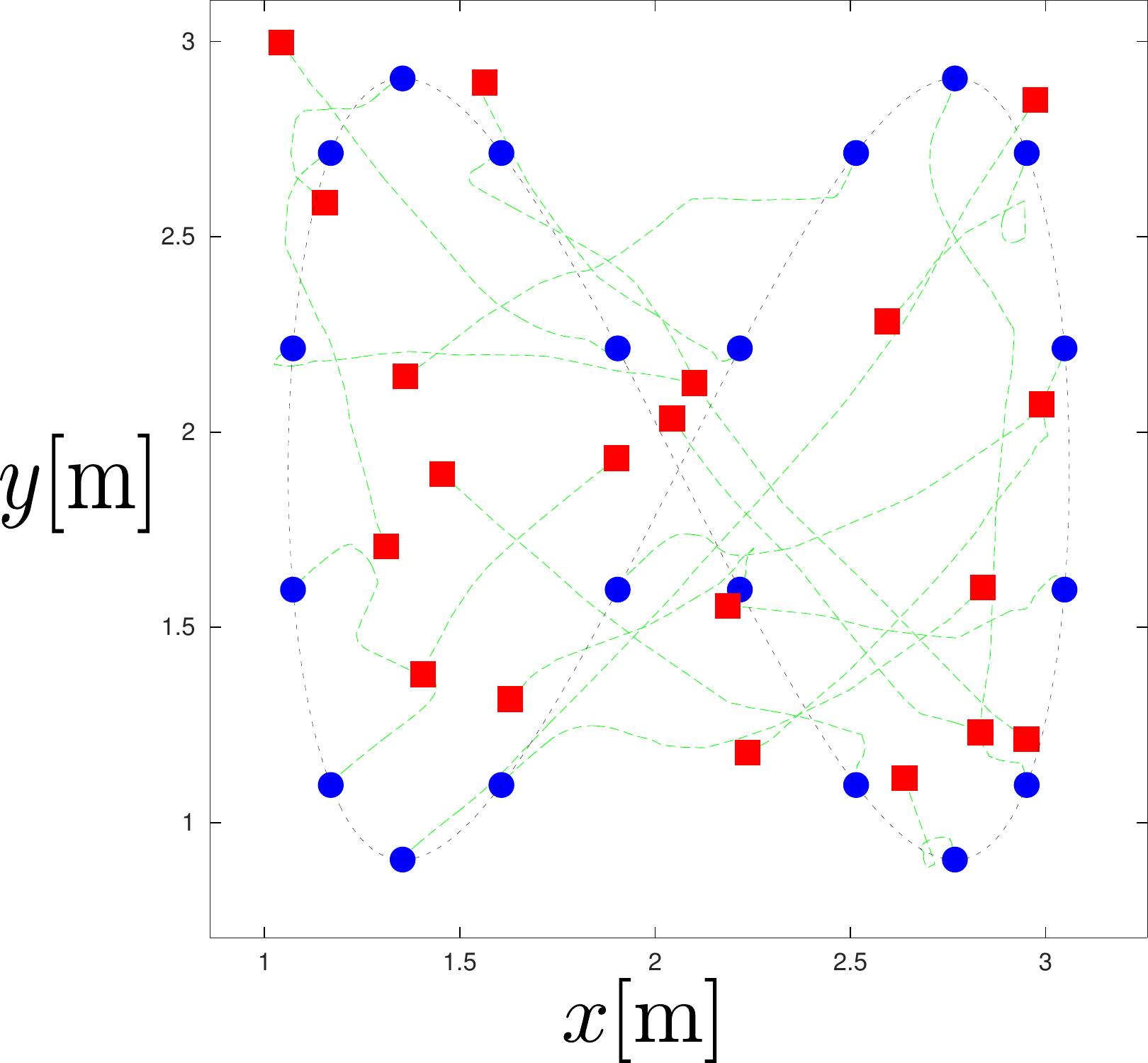}
\caption{Formation trajectories for the Example \ref{ExampleFormation}}\label{Fig:formationTraj}
\end{subfigure}
\end{center}
\end{figure}

\section{Conclusions and Future Work}
\label{Sec.Conclu}

A new class of consensus algorithms with predefined-time convergence have been introduced in this work. These results allow the design of a consensus protocol which solves the average consensus problem with predefined-time convergence, even under switching dynamic networks. A computationally simpler predefined-time consensus algorithm for fixed topologies was also proposed, with the trade-off that it does not converge to the average; moreover, an additional analysis proved that fixed-time convergence is also maintained under dynamic networks. These results were applied to the multi-agent formation control problem guaranteeing predefined-time convergence.
As future work, we consider the application of these results to provide predefined-time convergence to different consensus-based algorithms, such as distributed resource allocation~\cite{Xu2017,Xu2017b}.

\appendix
\section{Some useful inequalities}
In this appendix, we recall the inequalities used along the manuscript~\eqref{Th:ConsensusA} and~\eqref{Th:ConsensusB}. An interested reader may review~\cite{Mitrinovic1970,Hardy1934,Cvetkovski2012,Hardy1988}.



\begin{lemma}
\label{Lemma:Hardy}
\cite{Hardy1988}
Let $x=[x_1 \ \cdots \ x_n]^T\in\mathbb{R}^n$ and 
$$
\Vert x \Vert_p = \left(\sum_{i=1}^n|x_i|^p\right)^\frac{1}{p}
$$
then if $s>r>0$ then the following inequality holds
$$
\|x\|_s\leq\|x\|_r\leq n^{\frac{1}{r}-\frac{1}{s}}\|x\|_s.
$$
\end{lemma}



\end{document}